\documentclass[11pt]{amsart}

\usepackage[margin=3cm]{geometry}
\usepackage{amsmath,amsthm}
\usepackage{amsfonts,amstext,amssymb, mathtools, comment}
\usepackage{amsfonts,amstext,amssymb,tikz}
\usepackage{graphicx}
\RequirePackage[colorlinks,citecolor=blue,urlcolor=blue]{hyperref}

\newcommand{\levy}{L\'{e}vy }
\newcommand{\p}{{\mathbb P}}

\newcommand{\e}{{\mathbb E}}
\newcommand{\D}{{\mathrm d}}
\newcommand{\R}{{\mathbb R}}
\newcommand{\N}{{\mathbb N}}

\renewcommand{\a}{\alpha}

\newcommand{\ind}[1]{\mbox{\rm\large  1}_{\{#1\}}}
\newcommand{\indevent}{{\rm\large  1}}

\newcommand{\ep}{\varepsilon}
\newcommand{\epp}{{(\varepsilon)}}

\newcommand{\X}{{\widehat X}}
\newcommand{\V}{{\widehat V}}
\renewcommand{\O}{{\mathcal O}}

\DeclareMathOperator\RV{RV}

\newcommand{\ii}{{\mathrm i}}
\newcommand{\n}{{(n)}}

\newcommand{\Oh}{{\mathrm{O}}}
\newcommand{\oh}{{\mathrm{o}}}

\newcommand{\eqd}{\overset{{\rm d}}{=}}
\newcommand{\convd}{\overset{{\rm d}}{\to}}

\newtheorem{theorem}{Theorem}
\newtheorem{proposition}{Proposition}
\newtheorem{corollary}{Corollary}
\newtheorem{lemma}{Lemma}
\newtheorem{remark}{Remark}

\begin{document}
\title[Zooming-in on a L\'evy process]{Zooming-in on a L\'evy process:\\ Failure to observe threshold exceedance over a dense grid}
\author[K.\ Bisewski and J.\ Ivanovs]{Krzysztof Bisewski and  Jevgenijs Ivanovs}
\address{Centrum Wiskunde \& Informatica, the Netherlands}
\address{Aarhus University, Denmark}
\begin{abstract}
	For a L\'evy process $X$ on a finite time interval consider the probability that it exceeds some fixed threshold $x>0$ while staying below $x$ at the points of a regular grid. We establish exact asymptotic behavior of this probability as the number of grid points tends to infinity. We assume that $X$ has a zooming-in limit, which necessarily is $1/\a$-self-similar L\'evy process with $\a\in(0,2]$, and restrict to $\a>1$.
	Moreover, the moments of the difference of the supremum and the maximum over the grid points are analyzed and their asymptotic behavior is derived.
	It is also shown that the zooming-in assumption implies certain regularity properties of the ladder process, and the decay rate of the left tail of the supremum distribution is determined.
\end{abstract}

\maketitle

	\section{Introduction}\label{sec:intro}
	Consider a L\'evy process $X=(X_t,t\geq 0)$ on the real line and let 
	\[M=\sup\{X_t:t\in[0,1]\},\qquad \tau=\inf\{t\geq 0:X_t\vee X_{t-}=M\}\]
	be the supremum and its time, respectively, for the time interval $[0,1]$. 
	For any $n\in\mathbb N_+$ consider also the maximum of $X$ over the regular grid with step size~$1/n$:
	\[M^\n=\max\{X_{i/n}:i=0,\ldots,n \}.\]
	In this paper we derive exact asymtotic behavior of
	\begin{equation}\label{eq:delta_def}\Delta_n(x)=\p(M>x,M^\n\leq x)=\p(M>x)-\p(M^\n>x)\end{equation}
	as $n\to\infty$ for any fixed~$x>0$, which is the probability of failure in detecting threshold exceedance when restricting to the grid time-points.
	On the way towards this goal, we also provide asymptotics of the moments $\e(M-M^\n)^p$ of the discretization error in the approximation of the supremum, markedly improving on the bounds in~\cite{giles} and other works.
	
	The motivation comes from various applications, where it is vital to understand if the process $X$ has exceeded a fixed threshold $x>0$ or not.
	Application areas include, among others, insurance and mathematical finance (pricing of barrier options), energy science (electric load), environmental science (pollution levels and exposure) and computer reliability.
	Normally, we observe the process of interest over a dense regular grid without having full knowledge about the continuous-time trajectory.
	It is then natural to base our judgment on $M^\n$ instead of~$M$. Thus $\Delta_n(x)$ is the probability of  making an error: the process exceeds $x$ but not over the grid points. Furthermore, our limit result can be used to provide a correction to $\p(M>x)$ when approximating it by $\p(M^\n>x)$ derived from Monte Carlo simulation. Or vice versa, it can provide a correction to $\p(M^\n>x)$ in cases when the formula for $\p(M>x)$ is available, see e.g. \cite{broadie_glasserman_kou}.
	It is noted that $M^\n$ always underestimates $M$, and so one can also consider more accurate estimators (but also more complicated, since these are potentially based on all the available information), which we leave for future work.
	
	Our main vehicle is the zooming-in limit theory of~\cite{ivanovs_zooming}, where it is shown 
	under a weak regularity assumption, see~\eqref{eq:zooming} below, that 
	\begin{equation}\label{eq:supremum_error}
	\Big(V^\n\mid\tau\in(0,1)\Big)\convd \V,\qquad\text{ with }\qquad V^\n:=b_n(M-M^\n)
	\end{equation}
	for some specific sequence $b_n>0$ and a random variable~$\V$. More precisely, $\V$ is defined in terms of the law of a self-similar L\'evy process $\X$ (the limit under zooming-in),  see~\eqref{eq:V}, and $b_n$ is chosen such that $b_nX_{1/n}\convd \X_1$.
	
	The convergence in~\eqref{eq:supremum_error} readily suggests that $\e(M-M^\n)^p$ is of order $b_n^{-p}$ and supplements it with exact asymptotics, but only when the underlying sequence of random variables $(V^\n)^p$ is uniformly integrable.
	Establishing the latter, however, is far from trivial and we could only do that thanks to~\cite{bertoin_splitting} providing a representation of the pre- and post-supremum process using juxtaposition of the excursions in half-lines. Interestingly, the scaled moments would explode in some cases if we considered grid points to the right (or to the left) of~$\tau$ only. Furthermore, certain conditions must be fulfilled, and the decay of the moments can never be faster than $1/n$ if $X$ has jumps of both signs.
	
	The intuition behind the asymptotics of the detection error probability $\Delta_n(x)$ is given by the following:
	\begin{align*}
	b_n\Delta_n(x) &= b_n\p(x < M \leq x + b_n^{-1} V^\n)\\
	&\approx \int_0^\infty b_n\p(x < M \leq x + b_n^{-1} y)\p(V^\n \in \D y)\\
	&\to f_M(x)\e \V,
	\end{align*}
	where $f_M$ is a density of~$M$; we also show that $\e \V$ has a simple explicit formula in terms of the basic parameters. The second line is suggested by the asymptotic independence of $V^\n$ and $M$ (the convergence in~\eqref{eq:supremum_error} is R\'enyi-mixing).
	Note also that uniform integrability of $V^\n$ is needed in the last step, which forces us to assume that $X$ has unbounded variation on compacts; we assume that $\a\in(1,2]$, see~\eqref{eq:alpha}.
	Making $\approx$ precise turns out to be a major undertaking. In fact, asymptotic independence between $M$ and $V^\n$ is not enough -- one can construct counterexamples resembling those in~\cite[Theorem~2.4(ii)]{avikainen2009irregular}.
	In addition to exact asymptotics, we also provide bounds on both the moments and the error probability $\Delta_n(x)$ in the cases when the zooming-in assumption~\eqref{eq:zooming} is not satisfied. 
	
	{\it Paper structure}: 
	In Section~\ref{sec:prelim} we set up the notation and provide a thorough discussion of the basic zooming-in assumption~\eqref{eq:zooming} including some important classes of examples.
	Section~\ref{sec:ladder} explores some immediate implications of assumption~\eqref{eq:zooming}, including an invariance principle for the ladder process and regular variation of the associated bivariate exponents and L\'evy measures. This is then used to establish the rates of decay of $\p(M\leq \ep)$ and $\p(\tau\leq \epsilon)$ as $\ep\downarrow 0$. It is noted that the material in the other sections is largely independent of Section~\ref{sec:ladder}. In Section~\ref{sec:moments} we show uniform integrability of $V^\n$ for $\a>1$ and provide moment asymptotics of $\e (M-M^\n)^p$ for $\a>p$ under the zooming-in assumption~\eqref{eq:zooming}; we show a logarithmically tight upper bound otherwise. Section~\ref{sec:detection_error} establishes exact asymptotics of the error probability~$\Delta_n(x)$.
	The proofs of many preparatory and auxiliary results are deferred to Appendix~\ref{app:proofs}, while Appendix~\ref{app:correction} provides a correction to the main proof in~\cite{ivanovs_zooming} which is crucial for the developments in this work.
	
	{\it Literature overview}: The fundamental work in this area is~\cite{asmussen_glynn_pitman1995}, where the limit theorem for $M-M^\n$ was established in the case of a linear Brownian motion~$X$. This sparkled research in various application areas including mathematical finance, see~\cite{broadie_glasserman_steven} for approximations of option prices in discrete-time models using continuous-time counterparts.
	Various expansions and bounds on the expected error $\e(M-M^\n)$ were derived in~\cite{chen_thesis, giles, janssen_leeuwaarden} among others.
	The error probability $\Delta_n(x)$ asymptotics was identified in~\cite{broadie_glasserman_kou} in the case of a linear Brownian motion and later extended in~\cite{dia_lamberton2011} to Brownian motion perturbed by an independent compound Poisson process. 
	An interested reader may consult~\cite{dieker2017euler} for an overview of the literature regarding discretization of Brownian motion, see also~\cite{bisewski} for non-uniform grids. 
	Furthermore, there is a large body of literature in risk theory concerned with stochastic observation times, such as the epochs of an independent Poisson process, see~\cite{albrecher_ivanovs} providing a link between various exit problems for Poissonian and continuous observations. 
	
	There is also a large body of literature concerned with the supremum of a L\'evy process, see~\cite{mijatovic, chaumont_supremum,  suprema_bounds} among many others, and with the small-time behavior of L\'evy processes, see~\cite{bertoin_doney_maller, deng2015shift, doney_fluctuations, figueroa2008small} and references therein. 
	
	\section{Preliminaries and examples}\label{sec:prelim}
	To set up the notation, recall the L\'evy-Khintchine formula
	\[\e e^{\theta X_t}=e^{\psi(\theta)t}, \quad \psi(\theta)=\gamma\theta+\frac{\sigma^2}{2}\theta^2+\int_{\R}\left(e^{\theta x}-1-\theta x\ind{|x|<1}\right)\Pi(\D x)\]
	with $\theta\in \ii\R$ and parameters $\gamma\in\R,\sigma\geq 0,\Pi(\D x)$ where the latter is a L\'evy measure satisfying $\int_{\R} (x^2\wedge 1)\Pi(\D x)<\infty$. 
	In the case of $\int_{-1}^1|x|\Pi(\D x)<\infty$ we have a simplified expression
	\begin{equation}\label{eq:simplified}\psi(\theta)=\gamma'\theta+\frac{\sigma^2}{2}\theta^2+\int_{\R}\left(e^{\theta x}-1\right)\Pi(\D x)\end{equation}
	with $\gamma'\in\R$ being called the linear drift. 
	We write ub.v.\  and b.v.\ for processes of unbounded and bounded variation on compacts, respectively. Recall that b.v.\ case corresponds to $\sigma=0$ and $\int_{-1}^1|x|\Pi(\D x)<\infty$, so that~\eqref{eq:simplified} can be used.
	Furthermore, we let 
	\[\overline X_t=\{X_s:s\in[0,t]\},\qquad \underline X_t=\inf\{X_s:s\in[0,t]\}\] be the running supremum and infimum of $X$, respectively, so that $M=\overline X_1$.
	This notation will be needed in a few places below, where different time horizons and L\'evy processes are used. 
	
	Let us introduce some notation for the tails of $\Pi$:
	\begin{equation}\label{def:Pi_tails}
	\overline\Pi_+(x) = \Pi(x,\infty), \quad \overline\Pi_-(x) = \Pi(-\infty,-x), \quad \overline\Pi(x) = \overline\Pi_+(x) + \overline\Pi_-(x)
	\end{equation}
	with $x>0$.
	We also define the truncated mean and variance functions for $x\in(0,1)$:
	\begin{equation}\label{def:trunc_mv}
	m(x) = \gamma - \int_{x<|y|<1} y\Pi(\D x), \quad v(x) = \sigma^2 + \int_{|y|<x} y^2\Pi(\D x),
	\end{equation}
	which play a fundamental role in the study of small time behavior of~$X$.
	Finally, we write $f\in\RV_\a$ to say that $f$ is a function regularly varying at $0$ with index~$\a$, see~\cite{bingham_regular}. 
	
	\subsection{Important indices}
	Define the following indices:
	\begin{equation}\label{eq:beta}
	\begin{split}
	\beta_0&:=\inf\big\{\beta\geq 0:\int_{|x|<1} |x|^\beta\Pi(\D x)<\infty\big\},\\
	\beta_\infty&:=\sup\big\{\beta\geq 0:\int_{|x|>1} |x|^\beta\Pi(\D x)<\infty\big\}.
	\end{split}
	\end{equation}
	The index $\beta_0\in[0,2]$ provides some basic information about the intensity of small jumps and is often called Blumenthal-Getoor index, whereas $\beta_{\infty}\in[0,\infty]$ is about integrability of big jumps.
	Moreover, let
	\begin{align}\label{eq:alpha}
	\alpha=\begin{cases}2,&\sigma\neq 0\\1,& \text{b.v.\ with }\gamma'\neq 0\\\beta_0, &\text{otherwise}
	\end{cases}
	\end{align}
	and note that necessarily $\a\geq \beta_0$. 
	
	\subsection{Attraction to self-similar processes under zooming in}
	Throughout  most of this work we assume that 
	\begin{equation}\label{eq:zooming}
	X_{\ep}/a_\ep\convd \X_1,\qquad\text{ as }\ep\downarrow 0
	\end{equation}
	for some function $a_\ep>0$ and a random variable $\X_1$, not identically 0. Then necessarily~\cite[Thm.\ 15.12(ii)]{kallenberg} $\X_1$ is infinitely divisible, and the above weak convergence extends to the convergence of the respective processes (in Skorokhod $J_1$ topology): 
	\[(X_{\ep t}/a_\ep)_{t\geq 0}\convd (\X_t)_{t\geq 0}.\] Furthermore, the \levy process $\X$ must be self-similar with Hurst parameter $1/\alpha$ for some $\alpha\in(0,2]$, implying that it is either 
	\begin{itemize}
		\item[(i)] a (driftless) Brownian motion, and then $\a=2$,
		\item[(ii)] a linear drift, and then $\alpha=1$,
		\item[(iii)] a strictly $\a$-stable \levy process, and then $\alpha\in(0,2)$. 
	\end{itemize}
	Note that $\alpha=1$ corresponds to two different classes: drift processes and strictly 1-stable process also known as a Cauchy process (symmetric and drifted), and this suggests that  $\alpha=1$ is often an intricate case.
	Moreover, $a_\ep\in\RV_{1/\alpha}$, that is, the scaling function is regularly varying at ~0 with index $1/\a$. The respective domains of attraction are completely characterized in terms of L\'evy triplets in~\cite{ivanovs_zooming}, which also provides a comprehensive overview of the related literature (the domains of attraction to the Brownian motion and linear drift have been characterized in~\cite{doney_maller} before). We emphasize that the parameter $\alpha$ is necessarily of the form~\eqref{eq:alpha}, see~\cite[Cor.\ 1]{ivanovs_zooming}.
	Moreover, the limit $\X$ is unique up to scaling by a positive constant and $a_\ep$ is unique up to asymptotic proportionality: $a_\ep\sim ca'_\ep$. 
	
	In the following we will make extensive use of the positivity parameter of the attractor:
	\[\rho=\p(\X_1>0),\]
	which can be easily derived from the L\'evy triplet of $\X$ using formulas in~\cite{zolotarev}.
	It is well known that the pair $(\alpha,\rho)$ specifies the self-similar process~$\X$ up to scaling by a positive constant. This is sufficient for our purpose since we may always choose an appropriate scaling sequence $a_\ep$.
	For $\a\in(1,2]$ the range of $\rho$ is given by $\rho\in[1-1/\a,1/\a]$ with the left and right boundary values corresponding to the spectrally-positive and spectrally-negative processes, respectively, see Figure~\ref{fig:arho}.
	For $\a\in(0,1]$ we have $\rho\in[0,1]$ with boundary values corresponding to a decreasing and increasing processes, respectively. Finally, $\a=1,\rho\in(0,1)$ specifies the class of drifted Cauchy processes, whereas $\a=1,\rho=\pm 1$ corresponds to linear drifts with signs~ $\pm$.
	We often write 
	\[X\in\mathcal D_{\a,\rho}\] to say that~\eqref{eq:zooming} holds with $1/\a$-self-similar process $\X$ having positivity parameter~$\rho$.
	Note that $X\in\mathcal D_{\a,\rho}$ implies that $\p(X_\ep>0)\to \rho$ as $\ep\downarrow 0$, which follows from the fact that $\p(\X_1=0)=0$.
	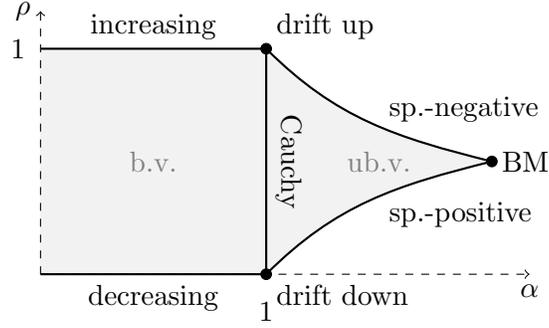
\begin{figure}
		\begin{center}
			\begin{tikzpicture}
			\filldraw[black!5] (0,0)--(0,3) -- (3,3) to[out=-45,in=166] (6,1.5) to[in=45,out=194] (3,0);
			\draw[->,dashed] (0,0) -- (6.5,0) node[anchor=north]{$\alpha$};
			\draw[->,dashed] (0,0) -- (0,3.5) node[anchor=east]{$\rho$};
			\draw[thick] (0,3) --  (3,3) to[out=-45,in=166] (6,1.5) to[in=45,out=194] (3,0)--(0,0);
			\filldraw (6,1.5) circle (2pt) node[anchor=west] {BM};
			\filldraw (3,0) circle (2pt) node[anchor=north west] {drift down};
			\filldraw (3,3) circle (2pt) node[anchor=south west] {drift up};
			\draw(1.5,3) node[anchor=south]{increasing};
			\draw(1.5,0) node[anchor=north]{decreasing};
			\draw(4.5,2.2) node[anchor=west]{sp.-negative};
			\draw(4.5,0.8) node[anchor=west]{sp.-positive};
			\draw[thick] (3,0)--(3,3);
			\node[rotate=-90] at (3.3,1.5){Cauchy};
			\draw[gray](1.5,1.5)node{b.v.};
			\draw[gray](4.5,1.5)node{ub.v.};
			\draw(-0.3,3)node{$1$};
			\draw(3,-0.5)node{$1$};
			\end{tikzpicture}
		\end{center}
		\caption{Self-similar L\'evy processes with parameters $(\a,\rho)$}	
		\label{fig:arho}
	\end{figure}
	
	\subsection{Examples}
	The trivial examples satisfying~\eqref{eq:zooming} are (i) $\sigma>0$ and arbitrary $\gamma,\Pi(\D x)$ and (ii) b.v.\ process with $\gamma'\neq 0$ and otherwise arbitrary $\Pi(\D x)$.
	In case (i) $\X$ is a Brownian motion and $a_\ep$ is asymptotically proportional to $\ep^{1/2}$, and in case (ii) $\X$ is a linear drift (having the same sign as $\gamma'$) and $a_\ep$ is asymptotically proportional to $\ep$.
	
	Let us stress that~\eqref{eq:zooming} is a weak regularity assumption satisfied for almost every L\'evy process of practical interest. The most notable exceptions are driftless compound Poisson process and its neighbors: driftless gamma and variance gamma processes, both of which satisfy $\sigma=\gamma'=0,\overline\Pi(x)\in\RV_0$ as $x\downarrow 0$. In other words, small jump activity is too weak to have a non-trivial limit.
	In the following we briefly consider some important classes of L\'evy processes and establish their zooming-in limits (such examples are missing in~\cite{ivanovs_zooming}).
	
	\subsubsection{Tempered stable processes (also known as CGMY)}
	
	A rather general class of \levy processes is obtained by considering the \levy measure $\Pi(\D x)$ of the form
	\[\Pi(\D x)=c_+x^{-1-\a_+}e^{-\lambda_+ x}\ind{x>0}\D x+c_-|x|^{-1-\a_-}e^{\lambda_- x}\ind{x<0}\D x,\]
	where $c_\pm,\lambda_\pm\geq 0,\a_\pm<2$ and $\a_\pm>0$ if $\lambda_\pm=0$.
	Particular examples are stable, gamma, inverse Gaussian and variance gamma processes.
	When $c_\pm=0$ we put $\alpha_\pm=0$. We assume that there is no Gaussian part ($\sigma=0$) since otherwise $\X$ is a Brownian motion, and that in b.v.\ case ($\a_+,\a_-<1)$ there is no linear drift ($\gamma'=0$) since otherwise $\X$ is a linear drift process.
	
	We have the following cases according to~\cite[Thm.\ 2]{ivanovs_zooming} (for clarity we avoid specifying $\rho$ in some cases):
	\begin{itemize}
		\item if $\a_\mp\leq \a_\pm\in (0,1)\cup(1,2)$ then $\X$ is a strictly $\a_\pm$-stable process; in the case of a strict inequality $\X$  has one-sided jumps of sign~$\pm$.
		\item if $\a_\mp<\a_\pm=1$ or $\a_+=\a_-=1$ with $c_\mp<c_\pm$ then $\X$ is a linear drift process of sign~$\mp$ (the sign might look counterintuitive at first, see Remark~\ref{rem:nef_drift}). 
		\item if $\a_+=\a_-=1$ and $c_+=c_-$ then $\X$ is a Cauchy process.
		\item if $\a_+,\a_-\leq 0$ then such a process does not have a non-trivial limit under zooming-in; the intensity of jumps is too small.
	\end{itemize}
	In particular, the gamma process has no limit, and the same is true for variance gamma processes with 0 mean (in these cases we have $\a_\pm=0$).
	
	\subsubsection{Generalized hyperbolic processes}\label{sec:hyperbolic}
	Another important family of \levy processes is a 5-parameter class of generalized hyperbolic L\'evy motions introduced by Barndorff-Nielsen~\cite{barndorff_nielsen} and advocated for financial use in~\cite{eberlein}. Note that this class includes normal inverse Gaussian processes. Generalized hyperbolic processes have no Gaussian component and their L\'evy measure possesses a density behaving as $C/x^2+\O(1/|x|)$ at 0 for some constant $C>0$, see~\cite[Prop.\ 2.18]{raible2000levy}. Thus such processes are ub.v.\ and satisfy $\overline\Pi_\pm\in\RV_{-1}$ with $\overline\Pi_+(x)/\overline\Pi_-(x)\to 1$ as $x\downarrow 0$.
	We also see that $x\overline\Pi_+(x)\to C$ and with a little more work we find that $\int_{x\leq |y|<1}y\Pi(\D y)$ has a finite limit. Hence according to~\cite[Thm.\ 2]{ivanovs_zooming} every generalized hyperbolic motion is attracted to a Cauchy process.
	
	\subsubsection{Subordination}
	Another popular way to construct a L\'evy process is by considering $X_t=Y_{S_t}$, where $Y$ and $S$ are two independent L\'evy processes and the latter is non-decreasing (a subordinator). In this case, it is sufficient to check~\eqref{eq:zooming} for the underlying processes:
	\begin{lemma}
		Suppose that $Y_\ep/y_\ep\convd \widehat Y_1, S_\ep/s_\ep\convd \widehat S_1$ as $\ep\downarrow 0$ with $y_t,s_t>0$ and non-trivial limits. Then~\eqref{eq:zooming} is satisfied with $\X_1=\widehat Y_{\widehat S_1}$ and $a_\ep=y_{s_\ep}$. 
	\end{lemma}
	\begin{proof}
		With obvious notation we have for any $\theta\in\ii\R$ that $\psi_Y(\theta/y_\ep)\ep\to\psi_{\widehat Y}(\theta)$ and a similar statement is true for~$S$.
		Now compute
		\[\psi(\theta/a_\ep)\ep=\psi_S(\psi_Y(\theta/y_{s_\ep}))\ep=\psi_S(\psi_{\widehat Y}(\theta)/s'_\ep)\ep\to \psi_{\widehat S}(\psi_{\widehat Y}(\theta)),\]
		where $s_\ep'\sim s_\ep$ and such a change is irrelevant for the limit.
	\end{proof}
	Letting $\a_Y,\a_S\in(0,1]$ be the corresponding indices, see~\eqref{eq:alpha}, we find that $\a=\a_Y\a_S$; one can see this by recalling that $a\in\RV_{1/\a}$.
	Note also that $\rho=\rho_{\widehat Y}$.
	Thus subordination allows for a direct description of the limiting process without the need to identify the L\'evy triplet of~$X$.
	For example, the normal inverse Gaussian must belong to $\mathcal D_{1,1/2}$. That is, the limit is a Brownian motion subordinated by the $1/2$-stable subordinator and this yields a Cauchy process as already observed in~\S\ref{sec:hyperbolic}.

	\subsubsection{Change of measure}
	Another important tool is the exponential change of measure, also known as the Esscher transform.
	It turns out that~\eqref{eq:zooming} is invariant under an arbitrary measure change. 
	In this regard we consider an absolutely continuous measure $\mathbb Q\ll\p$ with respect to the sigma algebra $\mathcal F_1=\sigma(X_t,t\in[0,1])$.
\begin{lemma}
Assume~\eqref{eq:zooming} and consider $\mathbb Q\ll\p$ on $\mathcal F_1$. Then $X_{\varepsilon}/a_\varepsilon$ under $\mathbb Q$ weakly converges to $\X_1$ as $\varepsilon\downarrow 0$.
\end{lemma}	
	\begin{proof}
	It is sufficient to show that the convergence in \eqref{eq:zooming} is R\'enyi-mixing:
	\begin{equation}\nonumber\e(f(X_{\varepsilon}/a_\varepsilon)L)\to \e f(\X_1)\e L\end{equation}
	for an arbitrary integrable random variable~$L$ and a bounded continuous function~$f$, since then we take $L$ to be the respective Radon-Nikodym derivative and note that $\e L=1$.
	According to~\cite[Prop.\ 2(D'')]{aldous} it is sufficient to establish the joint convergence
	\[(X_{\varepsilon}/a_\varepsilon,X_{t_1},\ldots,X_{t_n})\convd(\X_1,X_{t_1},\ldots,X_{t_n})\qquad\varepsilon\downarrow 0,\]
	where $\X_1$ is independent of~$X$. This statement follows easily by writing $X_{t_i}=(X_{t_i}-X_\varepsilon)+X_\varepsilon$ on the left hand side and using the independence of increments.
	\end{proof}

	\subsubsection{A sufficient condition}
	The following result provides an easy-to-check condition which implies~\eqref{eq:zooming}. It does not allow, however, to check attraction to Cauchy processes. Note that one may also check the following condition for $-X$ instead of~$X$.
	\begin{proposition}\label{prop:zoom_sufficient}
		Assume that $\sigma=0$, $\gamma'=0$ in b.v.\ case, and $\overline\Pi_+\in\RV_{-\a}$ with $\a\in(0,2]$.
		Then \eqref{eq:zooming} holds true in the following cases:
		\begin{itemize}
			\item $\a=2$ and $\liminf_{x\downarrow 0}\overline\Pi_+(x)/\overline\Pi_-(x)>0$,
			\item $\a=1$ and $\liminf_{x\downarrow 0}\overline\Pi_+(x)/\overline\Pi_-(x)>1$\\ (positive/negative linear drift limit according to b.v./ub.v.), 
			\item $\a\in(0,1)\cup(1,2)$ and $\lim_{x\downarrow 0}\overline\Pi_+(x)/\overline\Pi_-(x)\in(0,\infty]$.
		\end{itemize}
	\end{proposition}
	\begin{proof}
		See Appendix \ref{app:s_prelim} for the proof.
	\end{proof}
	\begin{remark}\label{rem:nef_drift}\rm
		In the case $\a=1$  in Proposition~\ref{prop:zoom_sufficient} we assume that the positive jumps are dominant and show that $\X$ is then a linear drift. One expects that this drift is positive, which is indeed true when $X$ is b.v. If, however, $X$ is ub.v.\ then the limiting drift process has a negative slope, which on an intuitive level can be explained by the standard construction of $X$ as the limit of compensated compound Poisson processes. It turns out that the compensating drift `wins' -- it determines the sign.
	\end{remark}
	
	Let us stress that the case $\a=1$ with $\overline\Pi_+(x)\sim\overline\Pi_-(x)\in\RV_{-1}$ does not guarantee~\eqref{eq:zooming}. In this case, according to~\cite[Thm.\ 2]{ivanovs_zooming}, we need to check that $m(x)/(x\overline\Pi(x))$ has a limit in $[-\infty,\infty]$ with $m(x)$ defined in~\eqref{def:trunc_mv}; finite limits correspond to $\X$ being Cauchy.
	In the following, we provide an example where the latter does not hold true and thus~\eqref{eq:zooming} is not satisfied.
	
	We consider an ub.v.\ L\'evy process with 
	\[\sigma=\gamma=0, \ \ \text{ and } \quad \overline\Pi_+(x)=(1+u(x))/x,\qquad  \overline\Pi_-(x) = 1/x\]
	for small enough $x$, where $u(x)=\sin(\log(-\log x))/\log x\to 0$.
	One can verify that $(1+u(x))/x$ is decreasing for small enough $x$ and so we have legal $\overline \Pi_\pm(x)$ functions, which are asymptotically equivalent and $\RV_{-1}$. Now compute
	\[m(x)=c+\int_x^h y \D \frac{u(y)}{y}=c'-u(x)-\int_x^hy^{-1}u(y)\D y,\]
	where $h>0$ is some small number and $c,c'$ are constants. The latter integral evaluates to $\cos(\log(-\log x))-c''$ which is an oscillating function as $x\downarrow 0$ and the same is true about $m(x)$. Finally, $x\overline \Pi(x)\to 2$ and we see that $m(x)/(x\overline\Pi(x))$ oscillates as well, which shows that~\eqref{eq:zooming} does not hold.

	\section{First implications of the attraction assumption}\label{sec:ladder}
	
	Various properties of a process $X$ are inherited by the attractor $\X$; we assume~\eqref{eq:zooming} throughout this section. There are, however, numerous exceptions which may look surprising at first.
	For example, if $X$ does not hit $(0,\infty)$ immediately (point $0$ is irregular for $(0,\infty)$ for the process~$X$) then also $\X$ does not hit $(0,\infty)$ immediately, implying that $\X$ is decreasing, i.e., $\rho=0$. It is not true in general, however, that if $X$ hits $(0,\infty)$ immediately then so does $\X$, nor it is true that b.v.\ (ub.v.) process is attracted to b.v.\ (ub.v.) process; some counterexamples are given in \cite[Section~4.2]{ivanovs_zooming}.
	
	Recall that $X$ creeps upward if $\p(X_{\tau_x}=x)>0$ for some $x>0$ (and then for all).  Let us note that  a self-similar L\'evy process creeps upward if and only if it is spectrally-negative $\a\rho=1$, which can be readily verified from~\cite[Thm.\ VI.19(ii)]{bertoin}.
	Importantly, the creeping property is preserved when taking the zooming-in limit, see Corollary~\ref{cor:creeping} below.
	
	Finally, let us note that the material of this section is not required for Section~\ref{sec:moments} and Section~\ref{sec:detection_error}, apart from a simple bound in Lemma~\ref{lem:density_sup}. Nevertheless, the results of this section, apart from being of independent interest, may be used to gain further insight into the main quantities, see, e.g. the discussion at the beginning of Section \ref{sss:big_jumps}
	

	\subsection{Convergence of ladder processes}
	Let $( L^{-1}_t, H_t)$ be the ascending ladder process with the Laplace exponent $\kappa(a,b)$: $\e \exp(-a L^{-1}_t-b H_t)=\exp(-\kappa(a,b) t)$ and the local time $ L_t$ is normalized so that $\kappa(1,0)=1$. 
	This is always possible, even when $X$ is a decreasing process in which case $k(a,b)=1$ (the ladder process stays at $(0,0)$ before it is killed at rate 1).
	We use the obvious notation $q_{L^{-1}},\gamma'_{ L^{-1}}$ and $\Pi_{ L^{-1}}(\D x)$ to denote the killing rate, drift  and L\'evy measure of the ladder time process; similar notation is used with respect to the ladder height process~$H$.
	We write $(\widehat{ L}_t^{-1},\widehat{ H}_t)$ and $\widehat{\kappa}(a,b)$ for the ascending ladder process corresponding to $\X$ and its Laplace exponent, respectively. 
	Furthermore, we consider It\^o measure $ n(\cdot)$ of the excursions from the supremum, i.e., of the process $\overline X_t-X_t$,
	and let $\underline n(\cdot)$ be the analogue for $-X$ process (excursions from infimum for the original process). 
	It is tacitly assumed that $X_t\in\D y$ under $n(\cdot)$ and $\underline n(\cdot)$ implies $t<\zeta$, where $\zeta$ is the lifetime of an excursion.
	We refer the reader to~\cite{bertoin} for the construction and basic properties of these objects.
	
	Let us explicitly state a simple consequence of~\eqref{eq:zooming} observed in~\cite{invariance_chaumont_doney} and~\cite{ivanovs_zooming}.
	We note that when $X$ does not hit $(0,\infty)$ immediately the ladder process is a bi-variate compound Poisson, and the following result would be useless.
	\begin{proposition}\label{prop:loctime}
		If $X$ hits $(0,\infty)$ immediately and satisfies~\eqref{eq:zooming} then with $c_\ep=1/\kappa(1/\ep,0)$ we have
		\[\left(\frac{ L_{tc_\ep}^{-1}}{\ep},\frac{ H_{tc_\ep}}{a_\ep}\right)_{t\geq 0}\convd (\widehat{ L}_t^{-1},\widehat{ H}_t)_{t\geq 0}\qquad \text{ as }\ep\downarrow 0,\]
		and in terms of Laplace exponents:
		\[\frac{\kappa(a/\ep,b/a_\ep)}{\kappa(1/\ep,0)}\to\widehat{\kappa}(a,b).\]
		In particular, $( L^{-1}, H)$ has a zooming-in limit $(\widehat{ L}^{-1},\widehat{ H})$ with a pair of scaling functions $(c^{-1}_\ep,a_{c^{-1}_\ep})$, where $c_\ep^{-1}$ is the inverse of $c_\ep$.
	\end{proposition}
	\begin{proof}
		The local time of $X_{t\ep}/a_\ep$ is given by $ L_{t\ep}/c_\ep$, and the corresponding scaling requirement then reads: $\kappa(1/\ep,0)c_\ep=1$ yielding $c_\ep=1/\kappa(1/\ep,0)$.
		Thus the inverse local time and the ladder height processes are given by
		\[ L^{-1}_{tc_\ep}/\ep,\qquad X_{ L^{-1}_{tc_\ep}}/a_\ep= H_{tc_\ep}/a_\ep.\] 
		Finally, the weak convergence is stated in e.g.~\cite[Appendix C]{ivanovs_zooming}.
	\end{proof}
	
	Various interesting results follow immediately from the convergence in Proposition~\ref{prop:loctime}, and the following expressions for the $1/\a$-self-similar L\'evy process $\X$ with positivity parameter~$\rho$:
	\begin{equation}\label{eq:kappa}
	\widehat{\kappa}(a,0)=a^\rho\qquad \widehat{\kappa}(0,b)=\widehat{\kappa}(0,1)b^{\alpha\rho},
	\end{equation} see e.g.~\cite[\S VIII.1]{bertoin}.
	This, in particular, means that $\widehat{ L}^{-1}$ and $\widehat{ H}$ are $\rho$- and $\a\rho$-self-similar, respectively, when $\rho\neq 0$. 
	The case $\rho=0$ corresponds to a process staying at $0$ before getting killed, that is, the limit is trivial.
	\begin{corollary}\label{cor:kappa}
		If $X\in\mathcal D_{\a,\rho}$ then $\kappa(1/\ep,0)\in\RV_{-\rho}$ and $\kappa(0,1/\ep)\in\RV_{-\a\rho}$. 
	\end{corollary}
	\begin{proof}
		We only show the second claim. For any $b>0$
		\[\frac{\kappa(0,1/(ba_\ep))}{\kappa(0,1/a_\ep)}\to \frac{\widehat{\kappa}(0,1/b)}{\widehat{\kappa}(0,1)}=b^{-\a\rho},\]
		which is also true for $\rho=0$.
	\end{proof}
	
	\begin{corollary}\label{cor:creeping}
		If $X\in\mathcal D_{\a,\rho}$ creeps upwards then necessarily $\a\rho=1$, i.e., $\X$ is spectrally negative.
	\end{corollary}
	\begin{proof}
		According to~\cite[Thm.\ VI.19(ii)]{bertoin}, $X$ creeps upward if and only if $\gamma'_{ H}>0$, but then 
		$\kappa(0,1/\ep)\sim \gamma'_{ H}\ep^{-1}$ which is $\RV_{-1}$. Corollary~\ref{cor:kappa} shows that necessarily $\a\rho=1$.
		
		It is noted that one can give an alternative proof based on the characterization of creeping processes in terms of their L\'evy triplets, see~\cite{vigon}, or~\cite[Sec.\ 6.4]{doney_fluctuations}, but such a proof is much longer.
	\end{proof}
	Similarly, $L^{-1}$ has a positive drift component if and only if $X$ does not hit $(-\infty,0)$ immediately, in which case $\rho=1$ implying that $\X$ is a subordinator. This is equivalent to $X$ not hitting $(-\infty,0)$ immediately, which can also be derived more directly.
	
	\begin{corollary}\label{cor:n_rv}
		Assume that $X\in\mathcal D_{\a,\rho}$ is not a decreasing process. Then
		\begin{align*}
		\quad \overline\Pi_{ L^{-1}}(\ep)\in\RV_{-\rho}, 
		\end{align*}
		except when $X$ does not hit $(-\infty,0)$ immediately ($\rho=1$).
		Moreover,
		\begin{align*}
		\quad \overline\Pi_{H}(\ep)\in\RV_{-\alpha\rho}, 
		\end{align*}
		except when $X$ creeps upward ($\a\rho=1$).
	\end{corollary}
	\begin{proof}
		Corollary~\ref{cor:kappa} tells us that 
		\[\kappa(1/\ep,0)-q_{ L^{-1}}=\int_0^\infty(1-e^{-x/\ep})\Pi_{ L^{-1}}(\D x)=\ep^{-1}\int_0^\infty e^{-x/\ep}\overline\Pi_{ L^{-1}}(x)\D x\]
		is $\RV_{-\rho}$. Tauberian theorem~\cite[Thm.\ 1.7.1, Thm.\ 1.7.2]{bingham_regular} implies that $\overline\Pi_{ L^{-1}}\in\RV_{-\rho}$ when $\rho\neq 1$.
		The case $\rho=1$ follows from Proposition~\ref{prop:loctime} and the specification of the domains of attraction in~\cite[Thm.\ 2(iii)]{ivanovs_zooming} together with Lemma~\ref{lem:index1}; recall we have excluded the case where $L^{-1}$ has a linear drift. The same arguments can be used to prove the second statement.
	\end{proof}
	
	%

	Next, we discuss the behavior of the entrance law at small times, which will come in use in Section~\ref{sec:detection_error}. 
	\begin{proposition}\label{prop:bound_excursion}
		Assume that $X\in\mathcal D_{\a,\rho}$ is not monotone and it hits $(-\infty,0)$ immediately.
		Then \[n(\ep<\zeta)\in\RV_{-\rho}\] and for any fixed $\delta>0$ there exist $f,g\in\RV_\rho$ such that
		\[f(\ep)\leq \underline n(X_\ep>\delta)\leq g(\ep)\]
		where the lower bound additionally requires that $X$ has large enough positive jumps. 
	\end{proposition}
	\begin{proof}
		It is known that $n(\ep<\zeta)=q_{L^{-1}}+\overline \Pi_{L^{-1}}(\ep)$, where $\zeta$ is the lifetime of the generic excursion. Thus the first result follows immediately from Corollary~\ref{cor:n_rv}.
		The bounds on $\underline n(X_\ep>\delta)$ are derived in Appendix~\ref{app:s_ladder}.  
	\end{proof}	
	
	Finally, we  remark that one can proceed further from here to study, e.g.,  convergence of It\^o excursions and meanders among other things. It is noted that in~\cite{doney_rivero}  the authors considered a L\'evy process with a zooming-out limit (classical regime) and proved various local limit results.

	\subsection{On the left tail of supremum}
	First, we consider the supremum over infinite time horizon $\overline X_\infty$, which in the case of a killed process corresponds to the time horizon which is independent and exponentially distributed. We introduce the potential measure of the ladder height process: 
	\[U_H[0,x]=\e\int_0^\infty \ind{H_t\leq x}\D t,\]
	where by convention the indicator is 0 when $t$ exceeds the killing time of~$H$.
	\begin{proposition}\label{prop:sup_tail}
		If $X\in\mathcal D_{\a,\rho}$ then $U_H[0,x]\in\RV_{\a\rho}$, where $X$ may be a killed process. If $X$ is ether killed or drifts to~$-\infty$ then 
		\[\p(\overline X_\infty\leq x)\in \RV_{\a\rho}\qquad\text{ as }x\downarrow 0.\] 
	\end{proposition}
	\begin{proof}
		From Corollary~\ref{eq:kappa} we know that the Laplace exponent of the ladder height process is regularly varying at~$\infty$ with index $\a\rho$, killing adds a constant to it and so it is still $\RV_{\a\rho}$.
		Note that the potential measure $U_H(\D x)$ satisfies $\int_0^\infty e^{-\theta x}U_H(\D x)=1/\kappa(0,\theta)$, and hence by Tauberian theorem~\cite[Thm.\ 1.7.1]{bingham_regular} $U_H[0,x]\in\RV_{\a\rho}$ and the first statement is proven. But then
		\[\p(H_\infty\leq x)=q_{ H}U_H[0,x]\in\RV_{\a\rho},\]
		where $q_{ H}>0$ is the respective killing rate. The second result now follows.
	\end{proof}
	In the case of a deterministic time interval $[0,1]$ we only  have the sandwich bounds:
	\begin{proposition}\label{prop:sup_tail_1}
		Assume that $X\in\mathcal D_{\a,\rho}$. If $X$ is not a subordinator then there exist  $f,g\in\RV_{\a\rho}$ such that 
		\[f(x)\leq \p(\overline X_1\leq x)\leq g(x),\] and otherwise $\p(\overline X_1\leq x)$ decreases faster than any power as $x\downarrow 0$.
	\end{proposition}
	\begin{proof}
		For the upper bound, observe that 
		\[\p(\overline X_{e_1}\leq x)=\int_0^\infty e^{-t}\p(\overline X_t\leq x)\D t\geq  \int_0^1 e^{-t}\p(\overline X_1\leq x)\D t=c\p(\overline X_1\leq x),\]
		but the left hand side is $\RV_{\a\rho}$ by Proposition~\ref{prop:sup_tail}.
		The lower bound is based on the bounds from~\cite{suprema_bounds} 
		and requires at present a number of tricks. We, therefore, present the proof of the lower bound in Appendix~\ref{app:s_ladder}. Note that taking a subordinator with a positive drift gives an example with $\p(\overline X_1\leq x)=0$ for small enough $x$ which is not regularly varying; other subordinators are discussed in Appendix~\ref{app:s_ladder}.
	\end{proof}
	It would be interesting to understand if  $\p(\overline X_1\leq \cdot)\in\RV_{\a\rho}$ as is the case for a strictly stable process. For the latter process, even the density of $\overline X_1$ is regularly varying at $0$, see~\cite{doney_savov}.
	The time of supremum $\tau$ on the interval $[0,1]$ has indeed regularly varying tails:
	\begin{proposition}\label{prop:tau_tail}
		Assume that $X\in\mathcal D_{\a,\rho}$ is not monotone. Then for the time interval $[0,1]$ we have
		\[\p(\tau\leq \ep)\in\RV_{\rho},\qquad\p(\tau\geq 1-\ep)\in\RV_{1-\rho}.\]
	\end{proposition}
	\begin{proof}
		We only prove the first statement since the second is analogous by time reversal, say. We may assume that $X$ hits $(0,\infty)$ immediately, because otherwise the statement is immediate.
		\cite[Thm.\ 6]{chaumont_supremum} provides us with the formula
		\[\p(\tau\leq \ep)=\int_0^\ep \underline n(s<\zeta)(n(1-s<\zeta)+\gamma'_{L^{-1}})\D s.\]
		Note that the second term under the integral has a positive limit as $s\downarrow 0$, and the first term is $\RV_{-(1-\rho)}$ according to the first result of  Proposition~\ref{prop:bound_excursion}. Thus the integrand is $\RV_{-(1-\rho)}$ and hence $\p(\tau\leq \ep)\in\RV_{\rho}$.
	\end{proof}

	\section{Moments of the discretization error}\label{sec:moments}
	Consider the error $M-M^\n$ of approximating the supremum of a L\'evy process over the interval $[0,1]$ by the maximum over the uniform grid with step size $1/n$.
	Our first result provides an upper bound on the moments of this discretization error for a general process $X$, not necessarily satisfying the zooming-in assumption \eqref{eq:zooming}.
	
	\begin{theorem}\label{thm:moments_bound}
		For any $p>0$ satisfying $\int_{|x|>1}|x|^p\Pi(\D x)<\infty$ and any $\epsilon>0$ we have
		\[\e(M-M^\n)^p=\begin{cases}\Oh(n^{-p/\a+\ep}),& p\leq \a,\\
		\Oh(n^{-1}),&p>\a\end{cases}\]
		as $n\to\infty$.
		
		Moreover, the bound can be strengthened to $\Oh(n^{-p/\a})$ in the boundary cases: (i) $p\leq \a=2$ and (ii) $p\leq \a=1$ and $X$ is b.v.
	\end{theorem}

	For $p=1$ the result in Theorem \ref{thm:moments_bound} is close to \cite[Theorem 5.2.1]{chen_thesis}; in the case of b.v.\ process $X$ our bound $\Oh(n^{-1})$ is slightly sharper.
	Importantly, the result \cite{chen_thesis} cannot be generalized in a straightforward fashion to $p\neq1$, since it crucially relies on Spitzer identity.
	Nevertheless, \cite{giles} provides some bounds for $p\neq 1$ in the particular case when $\sigma=0, \gamma'=0$ and $\Pi(\D x)$ has a density sandwitched between $c_1 |x|^{-1-\a}$ and $c_2 |x|^{-1-\a}$ for small~$|x|$. These bounds, however, have suboptimal rates (in the logarithmic sense) unless $p>2\a$ or $X$ is spectrally negative.
	
	Our main goal, however, is to provide exact moment asymptotics, which is possible under the regularity assumption~\eqref{eq:zooming}. In the following, $V^\n := b_n(M-M^\n)$, as defined in ~\eqref{eq:supremum_error} and $b_n=1/a_{1/n}$.
	
	\begin{theorem}\label{thm:UI} Assume that $X\in \mathcal D_{\a,\rho}$. Then for any $p\in(0,\alpha)$ satisfying $\int_{|x|>1}|x|^p\Pi(\D x)<\infty$ the sequence $\e\left(V^\n\right)^p$ is bounded. 
	\end{theorem}
	
	For completeness let us recall from~\cite{ivanovs_zooming} that 
	\begin{equation}\label{eq:V}
	\V=\min\{-\widehat\xi_{U+\mathbb Z}\},
	\end{equation}
	where $(\widehat\xi_t,t\in\R)$ is the limit of $\X$ over $[0,T]$ as seen from its supremum point as $T\to\infty$, and $U$ is an independent uniform random variable.
	The weak convergence in \eqref{eq:supremum_error} and Theorem \ref{thm:UI} immediately yield the uniform integrability of certain powers of $V^\n$. Combining this result with the limiting expression for $\e \V^\n$ in~\cite[Prop. 2]{asmussen_ivanovs_twosided}, we obtain the following result (recall the definition of $\beta_\infty$ in~\eqref{eq:beta}).
	
	\begin{corollary}\label{cor:expected}
		Let $X\in \mathcal D_{\a,\rho}$. Then for any positive $p<\a\wedge \beta_\infty$ we have
		\[\e\left(V^\n\right)^p\to\e \V^p\,\p(\tau\notin\{0,1\})\, \in(0,\infty)\qquad\text{ as }n\to \infty.\]
		In particular, for $\alpha>1$ and $\beta_\infty>1$:
		\[\e V^\n\to \e \V=-\zeta\left(\frac{\a-1}{\a}\right)\e \X_1^+\qquad\text{ as  }n\to\infty,\]
		where $\zeta$ is the Riemann zeta function. 	
	\end{corollary}
	It is noted that $\e \X_1^+$ has an explicit expression, see~\cite[Thm.\ 3]{zolotarev} or~\cite{asmussen_ivanovs_twosided}. 
	\begin{proof}
		Note that if $\tau\in\{0,1\}$ with positive probability then $(V^\n\mid\tau\in\{0,1\})=0$ a.s., because of the nature of discretization.
		Now the first result follows from the weak convergence in~\eqref{eq:supremum_error} and uniform integrability of $(V^\n)^p$, where the latter is a consequence of Theorem~\ref{thm:UI} applied with a slightly larger~$p$.
		
		Next, we note that $\a\in(1,2]$ implies that $X$ is ub.v.\ process and, in particular, $\p(\tau\in(0,1))=1$.
		Moreover, $\e\widehat V^\n\to \e \V$, where $\widehat V^\n$ corresponds to the discretization of~$\X$ which is in its own domain of attraction.
		The limit of $\e\widehat V^\n$ was obtained in~\cite[Prop. 2]{asmussen_ivanovs_twosided} using self-similarity and Spitzer's identity, see also~\cite{asmussen_glynn_pitman1995}. 
	\end{proof}
	
	\subsection{Comments and extensions}
	Note that Theorem~\ref{thm:moments_bound} is weaker than Corollary~\ref{cor:expected} (when the conditions of the latter are satisfied) providing the exact asymptotics of $\e(M-M^\n)^p$. In particular, we see that $\e(M-M^\n)^p$ is a sequence regularly varying at $\infty$ with index $-p/\a$, which is clearly upper-bounded by $n^{-p/\a+\epsilon}$ for large~$n$. 
	
	\begin{remark}\label{rem:p_larger_a}\rm
		Assuming that $X$ has jumps of both signs, it is not hard to see that we may choose $c,h>0$ such that $\p(M-M^\n>h)\geq cn^{-1}$. Hence for any $p>0$ we must have $\e (M-M^\n)^p\geq c'n^{-1}$. Note that this complements the case $p>\a$ in Theorem~\ref{thm:moments_bound} with a lower bound of the same order. Furthermore, when $X\in \mathcal D_{\a,\rho}$ and $p>\a$, we get that $\e \left(V^\n\right)^p\geq b_n^pc'n^{-1}\to \infty$, because $b_n^p$ is regularly varying at $\infty$ with index $p/\a>1$. 
		
		\noindent This question is more delicate for one-sided processes, and we have no complete answer here. In the case of no jumps ($X$ is a Brownian motion with drift) boundedness of exponential moments was established in~\cite{asmussen_glynn_pitman1995}. Furthermore, if $X$ is a b.v.\ spectrally-negative (-positive) process, then the error $M-M^\n$ is bounded from above by $|\gamma'|n^{-1}$, showing that $\e (M-M^\n)^p=\Oh(n^{-p})$, see also~\cite{giles}.
	\end{remark}
	
	\begin{remark}\label{rem:grid_shift}\rm
		Letting $V^\n_s$ be the analogue of $V^\n$ but for a shifted grid $(i+s)/n$ with $i\in\mathbb Z$ and all points in $[0,1]$, 
		we note that also $\{\e \big(V^\n_s\big)^p : n\geq 1,s\in[0,1)\}$ is bounded under the assumptions of Thm.~\ref{thm:UI}; the proof does not need any modifications. This readily implies (by letting $s\uparrow 1,s\downarrow 0$) that we may also exclude the endpoints from the standard grid without affecting the result of Thm.~\ref{thm:UI}.
	\end{remark}
	
	
	\subsubsection{Dealing with big jumps}\label{sss:big_jumps}
	If $\int_1^\infty x\Pi(\D x)=\int_{-\infty}^{-1} |x|\Pi(\D x)=\infty$ then necessarily $\e V^\n=\infty$ for all $n\geq 1$, and 
	the analogous statement is true for all $p>0$. When only the positive jumps, say, are non-integrable we may still arrive to an unbounded sequence of $\e V^\n$. The problem is that the discretization error obtained by looking to the right of the supremum time exclusively may not be bounded in expectation, even when jumps exceeding~1 in absolute value are discarded (this can be shown using the lower bound on the entrance law in Prop.~\ref{prop:bound_excursion}).
	
	Importantly, we may remove the condition $\beta_\infty>1$ on the absence of big jumps if we restrict to the event that no two big jumps are close to each other or to the endpoints of the interval, say.
	For this, let $T_1<T_2<\ldots$ be the  times of jumps exceeding $1$ in absolute value and let $N$ be their number in the interval $(0,1)$; we also put $T_0=0$ and $T_{N+1}=1$. Finally, define the event to be excluded:
	\begin{equation}\label{eq:A}A^\n=\{\exists i\in \{0,\ldots, N\}: T_{i+1}-T_i<1/n\}.\end{equation}
	It is well-known that $\p(A^\n)=\Oh(1/n)$ as $n\to\infty$, which is also easy to see using Slivnyak's formula from Palm theory; this observation will be used in the proof of Prop.~\ref{prop:passage}.
	\begin{proposition}\label{prop:UI_onejump}
		Let $X\in \mathcal D_{\a,\rho}$ with $\a>1$. Then the family $\widetilde V^\n=V^\n\indevent_{{A^\n}^c}$ is uniformly integrable, and $\e \widetilde V^\n\to \e \V$.
	\end{proposition}
	
	\begin{proof}
		The proof is given in Appendix \ref{app:s_moments}.
	\end{proof}
	
	\subsubsection{Conjecture for processes of bounded variation}
	Consider a b.v.\ process $X$ with $\int_{|x|>1}|x|\Pi(\D x)<\infty$.
	Recall that we have an upper bound on $\e(M-M^\n)$ of order $n^{-1}$, see Theorem~\ref{thm:moments_bound}, and a lower bound of the same order when $X$ has jumps of both signs, see Remark~\ref{rem:p_larger_a}.
	It is thus natural to ask if $n\e(M-M^\n)$ has a finite positive limit as $n\to\infty$.
	We conjecture the following:
	\begin{equation}\label{eq:cpp}n\e(M-M^\n)\to \tfrac{1}{2} |\gamma'|\cdot \p(\tau\in(0,1)) + \tfrac{1}{2}I,\end{equation}
	where
	\begin{multline*}I=:\int_0^1\iint_{x,y,u,v\geq 0} ((x-u) \wedge (y-v))_+ \\\Pi(\D x)\Pi(-\D y)\p(-\underline X_t\in \D u)\p(\overline X_{1-t}\in \D v)\D t.
	\end{multline*}
	Here the first term is suggested by the small-time behavior, and the second comes from the possibility of having  a jump up (of size $x$) before $\tau$ and a jump down (of size $-y$) after $\tau$, and no observations in-between. This result should not rely on the zooming-in assumption~\eqref{eq:zooming}, and may require quite a different set of tools for its proof. Proposition~\ref{prop:cpp} in Appendix~\ref{app:s_moments} demonstrates~\eqref{eq:cpp} in the simple case of a compound Poisson process with drift. Note also that $I=0$ in the case of one-sided jumps, and otherwise we can not expect $n(M-M^\n)$ to be uniformly integrable.
	
	\subsection{Proofs}
	The main idea is to consider a certain upper bound on $V^\n$ given in terms of fluctuation of the zoomed-in process $X^\n_t=X_{t/n}/a_{1/n},t\in[0,n]$ around its supremum time.
	Importantly, this upper bound is stochastically dominated by the diameter of the range of $X^\n$ on the unit interval, and the latter can be analyzed using standard techniques.
	This crucial domination result relies on Bertoin's representation of pre- and post-supremum processes in~\cite{bertoin} and it is given by Lemma~\ref{lem:bertoin} below.
	
	Recall that we write $\overline X_T$ and $\tau(T)$ for the supremum and its time over the interval $[0,T]$ where $T>0$.	
	
	\begin{lemma}\label{lem:bertoin}
		Consider $X$ on the interval $[0,T]$ for any fixed $T\geq 1$ and let 
		\[Z_T=\sup_{t\in[0,1]}\{(\overline X_T-X_{\tau(T)-t})\wedge(\overline X_T-X_{\tau(T)+1-t})\}\]
		with convention that $X_s=-\infty$ if $s \notin [0,T]$.
		Then $Z_T$ is first-order stochastically dominated by $Z_1$ and hence by $\overline X_1-\underline X_1$.
	\end{lemma}
	\begin{proof}
		According to Bertoin's~\cite{bertoin_splitting} representation of the joint law of post- and pre-supremum processes on the interval $[0,T]$ we have $Z_T\,\eqd\, Z'_T$, where
		\begin{equation}\label{eq:bertoin}Z'_T := \sup_{t\in[0,1]}\{X_t^\Uparrow\wedge -(X_{1-t}^\Downarrow)\},\end{equation}
		where $X_t^\Uparrow=Y^+_{a^+_t}$ and $X_{t}^\Downarrow=Y^-_{a^-_t}$ for some processes $Y^\pm$ (which do not depend on the choice of $T$) and $a^+_t,a^-_t$ being the right-continuous inverses of $A^+_t :=\int_0^t\ind{X_s>0}\D s$ and $A^-_t := \int_0^t\ind{X_s\leq 0}\D s$, respectively.  In particular, $X^\Uparrow$ and $X^\Downarrow$  jump into cemetery state at the times $A_T^+$ and $A^-_T$, respectively, which is the only dependence on the time horizon~$T$. We use the convention that the cemetery state in the above minimum is ignored so that minimum yields the other quantity.  Thus for increasing $T$, the deaths of processes $X^\Uparrow$ and $X^\Downarrow$ can occur only later and hence $X_t^\Uparrow\wedge -(X_{1-t}^\Downarrow)$ may only become smaller for each~$t$, and so $Z'_1 \geq Z'_T$ a.s. 
		Thus $Z_T$ is stochastically dominated by $Z_1$, but from a simple sample path consideration we have $Z_1=\overline X_1-\underline X_1$ a.s.\ concluding the proof.
	\end{proof}
	
	In the above Lemma it is crucial to consider pre- and post-supremum processes together, since any of them can die arbitrarily early whereas the sum of lifetimes is exactly~$T$. In particular, we can not conclude that $\sup_{t\in[0,1],t\leq \tau}\{\overline X_T-X_{\tau(T)-t}\}$ is dominated by the analogous quantity for $T=1$. In fact, the opposite is true. Furthermore, the results of this Section are false if we look only to the left (or to the right) of the time of supremum. This should not be confused with the removal of the observations at the endpoints as discussed in Remark~\ref{rem:grid_shift}.
	
	
	In the following we consider a family of L\'evy processes
	\begin{equation}\label{def:Xn}
	X^\n_t=b_n X_{t/n}
	\end{equation}
	with $b_n=1/a_{1/n}$ whenever~\eqref{eq:zooming} holds. Let $(\sigma^\n,\gamma^\n,\Pi^\n)$ be the corresponding L\'evy triplets. It is noted that $\Pi^\n(\D x) = \Pi(b_n^{-1}\D x)/n$,
	\begin{equation}\label{eq:n_triplet}
	\gamma^\n = \frac{b_n}{n}\Big(\gamma - \int_{b_n^{-1}\leq|x|<1} x\Pi(\D x)\Big) \quad {\rm and} \quad \sigma^\n = \sigma b_n/\sqrt n.
	\end{equation}
	\begin{lemma}\label{lem:levyp_conv}
		Let $X\in \mathcal D_{\a,\rho}$. Then $\gamma^\n, \sigma^\n,\int_{|x|\leq 1} x^2\Pi^\n(\D x)$ have finite limits as $n\to\infty$.
		Moreover, for any $p<\a$ such that $\int_1^\infty x^p\Pi(\D x)<\infty$ we have
		\[\int_1^\infty x^p\Pi^\n(\D x)\to \int_1^\infty x^p\widehat\Pi(\D x)<\infty.\]	
	\end{lemma}
	\begin{proof} 
		The first part is a direct consequence of~\cite[Eq.\ (19)--(21)]{ivanovs_zooming}, see also~\cite[Thm.\ 15.14]{kallenberg}. The second part relies on regular variation and conditions for the domains of attraction, and is given in Appendix~\ref{app:s_moments}.
	\end{proof}
	
	\begin{lemma}\label{lem:UIsup}
		If $\gamma^\n_+, \sigma^\n,\int_{|x|\leq 1} x^2\Pi^\n(\D x),\int_1^\infty x^p\Pi^\n(\D x)$ are bounded then so is $\e(\overline X^\n_1)^p$.
	\end{lemma}
	\begin{proof}
		The following arguments seem to be rather standard.
		Let $Z^\n_t$ be the process $X^\n_t$, when the jumps exceeding~$1$ in absolute value are discarded and then the mean is subtracted; in other words we temporarily assume that $\Pi^\n(-\infty,-1)=\Pi^\n(1,\infty)=\gamma^\n=0$. Assume for the moment that $p>1$ and note that $x^p\leq ae^{x}$ for some $a>0$ and all $x\geq 0$. Hence
		\[\e\left|Z_1^\n\right|^p/a\leq \e\exp\left(\left|Z_1^\n\right|\right)\leq  \e\exp\left(Z_1^\n\right)+\e\exp\left(-Z_1^\n\right)\]
		showing that $\lVert Z_1^\n\rVert_p$ is bounded if so are $\psi_{Z^\n}(\pm 1)$, but the latter follows from the L\'evy-Khintchine formula and boundedness of $\sigma^\n,\int_{|x|\leq 1} x^2\Pi^\n(\D x)$.
		The process $Z^\n$ is a martingale, and by Doob's martingale inequality we have
		\[\lVert\overline Z_1^\n\rVert_p\leq \frac{p}{p-1}\lVert Z_1^\n\rVert_p.\]
		For $p\in(0,1]$ we simply use the inequality $x^p<1+x^2$ for all $x>0$, and so $\e (\overline Z_1^\n)^p$ is bounded for all $n\geq 1$.
		
		Note that 
		\[\overline X_1^\n\leq \overline Z_1^\n+\gamma^\n_++P_1^\n,\]
		where $P_1^\n$ is an independent compound Poisson process with L\'evy measure $\Pi^\n(\D x)\ind{x\geq 1}$.
		But for any $p>0$ we have the inequality $(x+y)^p\leq (1\vee 2^{p-1})(x^p+y^p)$  for all $x,y>0$.
		Thus it is left to show that $\e(P_1^\n)^p$ is bounded.
		By Minkowski inequality we find for $p\geq 1$ that
		\begin{align}\label{eq:poiss_minkowski}
		\lVert P_1^\n\rVert_p^p & \leq \lVert N^\n\rVert_p^p \lVert \Delta^\n\rVert_p^p \leq \e (N^\n)^{\lceil p \rceil}\int_1^\infty x^p\Pi^\n(\D x)/\lambda^\n
		\end{align}
		where $N^\n$ is Poisson  with intensity $\lambda^\n=\Pi^\n(1,\infty)$, and the generic jump $\Delta^\n$ is distributed according to $\ind{x>1}\Pi^\n(\D x)/\lambda^\n$. But the moments of Poisson distribution are polynomial functions (with 0 free term) of its intensity. This shows that the right hand side of \eqref{eq:poiss_minkowski} is indeed bounded, because so are $\Pi^\n(1,\infty)\leq \int_1^\infty x^p\Pi^\n(\D x)$.
		For $p\in(0,1)$ we use the simple bound: $(x+y)^p\leq x^p+y^p$ for all $x,y>0$.
	\end{proof} 
	
	\begin{proof}[Proof of Theorem~\ref{thm:UI}]
		Observe that $V^\n$ is the error made by considering the maximum of $X^\n_t$ at the times $0,1,\ldots,n$ as compared to its supremum on $[0,n]$.
		According to Lemma~\ref{lem:bertoin} we find that $V^\n$ is first-order stochastically dominated by $\overline X^\n_1-\underline X^\n_1$; it is sufficient to look at the discretization epochs right next to the time of supremum.
		But from Lemma~\ref{lem:levyp_conv} and Lemma~\ref{lem:UIsup} (applied to $X$ and $-X$) we readily see that the sequence $\e(\overline X^\n_1-\underline X^\n_1)^p,n\geq 1$ is bounded
		given that $p<\alpha$ and $\int_{|x|>1}|x|^p\Pi(\D x)<\infty$.
	\end{proof}

	\begin{proof}[Proof of Theorem~\ref{thm:moments_bound}]
		Define a family of L\'evy processes $X^\n_t$ as in \eqref{def:Xn} with $b_n=n^{1/\a_+}$ where $\a_+>\a\geq p$.
		Using Lemma~\ref{lem:bertoin} we obtain
		\[\e(M-M^\n)^p=n^{-p/\a_+}\e(b_n(M-M^\n))^p\leq n^{-p/\a_+}\e \left(\overline X^\n_1-\underline X^\n_1\right)^p.\]
		Furthermore, we may apply Lemma~\ref{lem:UIsup} because according to Lemma~\ref{lem:gen_conv0} in Appendix~\ref{app:s_moments}, all the relevant quantities have $0$ limits. The proof is now complete for $p\leq \a$. In the cases (i) $p\leq \a=2$ and (ii) $p\leq \a=1$ with $X$ b.v.\ we use Lemma~\ref{lem:gen_conv2} instead.
		
		When $p>\a$ we simply take $\a_+=p$ so that $b_n=n^{1/p}$. Note that 
		\[\int_1^\infty x^p\Pi^\n(\D x)=\int_{b_n^{-1}}^\infty x^p\Pi(\D x),\]
		which is bounded since $p>\a\geq \beta_0$. The above reasoning now applies and we get the upper bound of order $n^{-1}$.
	\end{proof}

	\section{Asymptotic probability of error in threshold exceedance}\label{sec:detection_error}
	
	Consider the error probability $\p(M>x,M^\n\leq x)$ in detection of threshold exceedance.
	The main aim of this section is to prove the following theorem.
	
	\begin{theorem}\label{thm:passage}
		Assume that $X\in\mathcal D_{\a,\rho}$ with $\a>1$. Then $M$ has a continuous density, say $f_M(x)$, and for any $x>0$
		\[b_n\p(M>x,M^\n\leq x)\to f_M(x)\e \V\]
		as $n\to \infty$, where $b_n=1/a_{1/n}$ and $\e \V$ is given in Corollary~\ref{cor:expected}.  
	\end{theorem}
	
	The intuition behind this result is explained in Section~\ref{sec:intro}. It is also noted that Theorem \ref{thm:passage} has been established for a linear Brownian motion in~\cite{broadie_glasserman_kou}, and later extended to an independent sum of a linear Brownian motion and a compound Poisson process in~\cite{dia_lamberton2011}.
	
	\begin{remark}\rm
		The result of Theorem~\ref{thm:passage} is also true for a shifted grid $(i+s)/n$ with all points in $[0,1]$.
		Furthermore, with some additional effort one can show that the limit in Theorem~\ref{thm:passage} holds uniformly in all positive levels $x$ away from $0$ and all shifts $s\in[0,1)$.
	\end{remark}
	
	\subsection{Preparatory results }
	We start by recalling a basic result concerning convergence of integrals. It follows immediately from the generalized continuous mapping theorem~\cite[Thm.\ 4.27]{kallenberg} and convergence of means~\cite[Lem.\ 4.11]{kallenberg}.
	\begin{lemma}[Convergence of integrals]\label{lem:vehicle}
		Consider a sequence of probability measures $\mu_n$ on a metric space $S$ with a weak limit $\mu$, and a sequence of measurable functions $h_n$ on $S$ such that $h_n(x_n)\to h(x)$ whenever $x_n\to x \in C\subset S$ such that $\mu(C)=1$. If, moreover,  the random variables $h_n(\mu_n)$, corresponding to the push-forward measures, are uniformly integrable then 
		\[\int h_n(x)\mu_n(\D x)\to \int h(x)\mu(\D x).\]
	\end{lemma}
	
	Throughout this section we assume that 
	\begin{equation}\label{eq:assumption_passage}X\in\mathcal D_{\a,\rho}\qquad \text{ with }\a>1,\end{equation} 
	which implies the Orey's condition:
	\begin{equation}\label{eq:orey}\liminf_{\epsilon\downarrow 0}\epsilon^{\gamma-2}\left(\sigma^2+\int_{-\epsilon}^\epsilon x^2\Pi(\D x)\right) >0\end{equation}
	for some $\gamma\in(1,2]$, since the function in brackets must be $\RV_{2-\a}$ according to~\cite{ivanovs_zooming} and then one can take $\gamma\in(1,\alpha)$; $\gamma=2$ is possible only when $\sigma^2>0$.
	Therefore, $X_t$ has a smooth bounded density, say $p(t,x)$, for each $t>0$, see~\cite[Prop.\ 28.3]{sato} and ~\cite[Thm.\ 3.1]{picard}.
	Moreover, $p(t,x)$ is continuous and strictly positive  on $(0,\infty)\times \R$, see~\cite{sharpe1969zeroes}. The following Lemma~\ref{lem:bounded_density} shows that $p(t,x)$ must be bounded on any set away from the origin. We could not locate such a result in the literature, but see~\cite[Prop.\ III.6]{leandre} and \cite[Prop.\ 6.3]{figueroa} for some related results.
	\begin{lemma}\label{lem:bounded_density}
		Assume that \eqref{eq:orey} holds with $\gamma>1$.
		Then for any $\delta>0$ the function $p(t,x)$ is upper bounded for all $t>0,x\in\R$ such that $t>\delta$ or $|x|>\delta$. 
		The conclusion may fail for any $\gamma<1$.
	\end{lemma}
	\begin{proof}
		See Appendix~\ref{app:detection_error}.
	\end{proof}
	It is noted that the proof of Lemma~\ref{lem:bounded_density} also shows that $p(0+,x)=0$ for any $x\neq 0$, when $\gamma>1$ and that for $\gamma<1$ this does not need to be the case. 
	In the following we write $\p_z$ for the law of the shifted L\'evy process $X$ with $X_0=z$.
	
	\begin{lemma}\label{lem:density}
		Assume that~\eqref{eq:orey} holds with $\gamma>1$. Then, for any $t>0$ the measure $\p_z(X_t\in \D x,\underline X_t>0)$ has a continuous density $f_{z,t}(x)$ which is bounded and jointly continuous on $\{(x,z):x>\delta,z>0\}$ for any $\delta>0$.
	\end{lemma}
	\begin{proof}
		We start as in the proof of~\cite[Lemma 8]{doney_savov}.
		By the strong Markov property applied at $\tau_0=\inf\{t\geq 0:X_t<0\}$, the first time the process becomes negative, we find  that
		\begin{align*}
		&\p_z(X_t\in\D x,\underline X_t\leq 0)/\D x=\int_{s\in(0,t),y\geq 0} \p_z(\tau_0\in\D s, -X_s\in\D y)p(t-s,x+y),
		\end{align*}
		where  we use $\p_z(\underline X_t=0)=0$ and $\p_z(\tau_0=t)=0$. 
		Note that it is enough to establish that the right hand side is jointly continuous for $x> \delta,z>0$, because then
		\[f_{z,t}(x)=p(t,x-z)-\int_{s\in(0,t),y\geq 0} \p_z(\tau_0\in\D s, -X_s\in\D y)p(t-s,x+y)\]
		must be bounded and jointly continuous.
		
		For any $z_n\to z>0,x_n\to x>\delta$ we need to show that
		\begin{align*}
		\int_{s\in(0,t),y\geq 0} \p_{z_n}(\tau_0\in\D s, -X_s\in\D y)p(t-s,x_n+y)
		\end{align*}
		has the corresponding limit.
		This readily follows from Lemma~\ref{lem:vehicle}, joint continuity of $p(t,x)$ and the fact that it is bounded for all $t>0$ and $x$ away from $0$, see Lemma~\ref{lem:bounded_density}.
	\end{proof}
	
	\subsection{Proofs}
	Our proof of Theorem \ref{thm:passage} essentially consists of two parts. First, we analyze a restricted problem when $\tau$ is away from~0 in Proposition~\ref{prop:passage}, which turns out to be much simpler than the original problem. 
	The main idea here is to split the path at some small $\delta$ and to time-reverse the first piece, in order to sandwich the probability of interest between the integrals converging to the right quantity.
	The original problem is then tackled by showing that the contribution arising from $\tau<\delta$ can be neglected as $\delta\downarrow 0$, which requires various further ideas.
	\begin{proposition}\label{prop:passage}
		Assume~\eqref{eq:assumption_passage} and chose $\delta\in(0,1)$. Then $\p(M\in \D x,\tau\geq \delta)$ has a continuous density, say $f_M(x;\delta)$, and
		\[b_n\p(M>x,M^\n<x,\tau\geq \delta)\to f_M(x;\delta)\e \V\]
		as $n\to \infty$ for any $x>0$.
	\end{proposition}
	\begin{proof}
		The upper bound on $\p(M>x,M^\n<x,\tau\geq \delta)$ is obtained by restricting the discretization grid to the times exceeding $\delta$, see Figure~\ref{fig:prop9}. 
		\begin{figure}[h!]
			\begin{center}
				\begin{tikzpicture}
				\draw (-1.2,0.5)--(-1,0.7)--(-0.8,-0.16)--(-0.6,0.34)--(-0.4,1.92)--(-0.2,1.77)--(0,0) -- (0.2,-0.05)--(0.4,-0.29)--(0.6,-0.88)--(0.8,1.5) -- (1,0.87)--(1.2,2.5)--(1.4,1.67)--(1.6,1.32)--(1.8,0.32) --(2,0.28) -- (2.2,-0.52)--(2.4,-2.05)--(2.6,-1.93)--(2.8,-1.46) -- (3,0.18)--(3.2,-0.81)--(3.4,-0.47)--(3.6,-0.34)--(3.8,0.3);
				\draw[->,thick] (-1.2,0.5)--(4,0.5);
				\draw[->,thick] (-1.2,0)--(-1.2,2.7);
				\draw (0,0.5) node[above right]{$\delta$};
				\draw (3.8,0.5) node[above]{$1$};
				\draw[->,thick,dashed] (0,2.5)--(-1.5,2.5);
				\draw[->,thick,dashed] (0,2.5)--(0,-1);
				\draw (0,0) node[left]{$z$};
				\draw (0,1.5) node[right]{$y$};
				\draw (-0.1,1.5)--(0.1,1.5);
				\draw (-0.1,0)--(0.1,0);
				\filldraw (0.8,1.5) circle (0.05);
				\filldraw (1.2,2.5) circle (0.05);
				\draw (-0.2,1.77) circle (0.05);
				\foreach \x in {-1.2,-0.2}   \draw (\x,0.6)--(\x,0.4);
				\foreach \x in {0.8,1.8,2.8,3.8}   \draw[thick] (\x,0.6)--(\x,0.4);
				\end{tikzpicture}
			\end{center}
			\caption{Schematic sample path and the reversal}
			\label{fig:prop9}
		\end{figure}
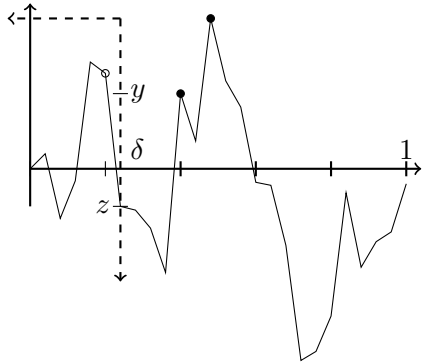
		Considering the post-$\delta$ process $X_{t+\delta}-X_\delta,t\geq 0$ (independent of $X_t,t\leq \delta$ and having the same law) and
		its functionals
		\[M_{[\delta,1]}:=\sup_{t\in[\delta,1]} X_t-X_\delta,\qquad V^\n_{[\delta,1]}/b_n:=\sup_{t\in[\delta,1]} X_t-\max_{i/n\in[\delta,1]} X_{i/n}\] we have the upper bound:
		\begin{align}\label{ali:ub}
		&\int_{z>0,y> 0} \p(M_{[\delta,1]}\in\D z,V^\n_{[\delta,1]}/b_n\in \D y)\p(x\vee \overline X_\delta<X_\delta+z<x+y)\nonumber\\
		=&\int_{z>0,y> 0} \p(M_{[\delta,1]}\in\D z,V^\n_{[\delta,1]}\in \D y)\p_z(X_\delta\in (x,x+y/b_n],\underline X_\delta>0)
		\end{align}
		where in the second line we have used the time and space reversal (yielding a process with the same law) at the time~$\delta$, see the dashed axes in Figure~\ref{fig:prop9}. 
		The lower bound is obtained by restricting to the event when the supremum over $[0,\delta)$ is smaller than the discretized maximum over $[\delta,1]$, implying that discretization epochs before $\delta$ do not matter. 
		Thus the lower bound is given by~\eqref{ali:ub} with a single change, where $\underline X_\delta>0$ is replaced by $\underline X_\delta>y/b_n$.
		
		
		According to~\cite{ivanovs_zooming} the measure $\p(M_{[\delta,1]}\in\D z,V^\n_{[\delta,1]}\in \D y)$ has the weak limit $\p(M_{[\delta,1]}\in\D z)\times \p(\V\in \D y)$, because we are discretizing the process with the same law and the limit does not depend on the time horizon neither on the grid shift.
		Moreover, for any positive $(y_n,z_n)\to(y,z)$ with $z,y>0$ the mean value theorem and Lemma~\ref{lem:density} show that 
		\[b_n\p_{z_n}(X_\delta\in (x,x+y_n/b_n],\underline X_\delta>0)=f_{z_n,\delta}(x_n)y_n\to yf_{z,\delta}(x),\]
		where $x_n\in (x,x+y_n/b_n)$. Moreover, the same limit holds true for 
		\begin{multline*}b_n\p_{z_n}(X_\delta\in (x,x+y_n/b_n],\underline X_\delta>y_n/b_n)\\=b_n\p_{z_n-y_n/b_n}(X_\delta\in (x-y_n/b_n,x],\underline X_\delta>0)\end{multline*}
		appearing in the lower bound.
		Assume $\beta_\infty>1$, see~\eqref{eq:beta}. Now Lemma~\ref{lem:vehicle} applies, because the uniform integrability of the corresponding family of measures follows from that of $V^\n_{[\delta,1]}$ and the boundedness of $f_{z,\delta}(x)$. Hence the limit of interest is 
		\[\int y f_{z,\delta}(x)\p(M_{[\delta,1]}\in\D z)\times \p(\V\in \D y)=f_M(x;\delta)\e \V,\]
		where the continuity of $\int f_{z,\delta}(x)\p(M_{[\delta,1]}\in\D z) = f_M(x;\delta)$ follows from the boundedness of $f_{z,\delta}(x)$ and the dominated convergence theorem.
		
		For $\beta_\infty\leq1$ note that $b_n\p(A^\n)=b_n\Oh(1/n)\to 0$, where the event $A^\n$ is defined in~\eqref{eq:A}. 
		Thus we may work on the event ${A^\n}^c$ for the post-$\delta$ process, and apply Proposition~\ref{prop:UI_onejump} to get uniform integrability.
	\end{proof}
	
	The following result, in particular, establishes continuity of the supremum density, see also~\cite{chaumont_malecki}.  
	\begin{lemma}\label{lem:density_sup}
		Assume \eqref{eq:assumption_passage}. Then $f_M(x):=\lim_{\delta\downarrow 0}f_M(x;\delta)$ is a density of $M$ continuous for $x>0$.
	\end{lemma}
	\begin{proof}
		According to~\cite{chaumont_supremum},
		\[f_M(x;\delta)=\int_{s\in(\delta,1),y>0}\underline n(X_{s/2}\in\D y)f_{y,s/2}(x)n(1-s<\zeta)\D s,\]
		which is also true for $\delta=0$ yielding $f_M(x)$; here $\zeta$ denotes the lifetime.
		
		It is left to show that $f_M(x)$ is continuous, for which it is sufficient to establish that 
		\[\int_{s\in(0,\delta),y>0} \underline n(X_{s}\in\D y)f_{y,s}(x)n(1-2s<\zeta)\D s\to 0\]
		as $\delta\downarrow 0$ uniformly in $x\geq x_0>0$, because $f_M(x;\delta)$ is continuous according to Proposition~\ref{prop:passage}.
		Recall from the proof of Lemma~\ref{lem:density} that $f_{y,s}(x)\leq p(s,x-y)$ and the latter is bounded when $x-y$ is away from $0$, see Lemma~\ref{lem:bounded_density}. Hence it is sufficient to show that 
		\[\int_0^{\delta} \underline n(X_{s}>x_0/2)\sup_xp(s,x)\D s\to 0,\]
		where $\sup_xp(s,x)=\Oh(s^{-1/\gamma})$ with $\gamma\in(1,\a)$ according to~\cite[Thm.\ 3.1]{picard}.
		But $\underline n(X_{s}>x_0/2)$ is upper bounded by a function in $\RV_\rho$ as $s\downarrow 0$  according to Proposition~\ref{prop:bound_excursion}, implying that it is bounded for small $s$ (converges to 0) since necessarily $\rho>0$. The proof is now complete.
	\end{proof}

	\begin{proof}[Proof of Theorem~\ref{thm:passage}]
		We need to show that
		\begin{equation}\label{eq:toshow}\limsup_{n\to\infty}b_n\p(M>x,M^\n<x,\tau<\delta)< C_x(\delta),\end{equation}	
		where $C_x(\delta)\downarrow 0$ as $\delta\downarrow 0$, because then
		Proposition~\ref{prop:passage} shows that
		\begin{multline*}f_M(x;\delta)\e \V\leq\liminf b_n\p(M>x,M^\n<x)\\
		\leq \limsup b_n\p(M>x,M^\n<x)\leq f_M(x;\delta)\e \V+C_x(\delta)\end{multline*}
		implying the result with the help of Lemma~\ref{lem:density_sup}.
		
		We assume that $\beta_\infty>1$, see~\eqref{eq:beta}, since the other case can be handled in exactly the same way as in Proposition~\ref{prop:passage}.
		In order to remove the effect of shifting the grid (needed later) we observe that 
		\[\p(M>x,M^\n<x,\tau<\delta)\leq \p(M>x,\underline M^\n<x,\tau<\delta),\]
		where 
		\[\underline M^\n=\inf_{t\in[0,1/n]}\{X_{\tau-t}\vee X_{\tau-t+1/n}\}\leq M^\n\]
		with the convention that $X_t=-\infty$ if $t\notin [0,1]$. Recall from the proof of Theorem~\ref{thm:UI} that $b_n(M-\underline M^\n)$ is uniformly integrable, see also Lemma~\ref{lem:bertoin}. Moreover, it can be shown that $b_n(M-\underline M^\n)$ has a weak (R\'enyi) limit, call it $\underline \V$, which corresponds to taking the same map of the limiting process seen from the supremum, see~\cite{ivanovs_zooming}; the limiting process is composed of two independent pieces neither of which can jump at a fixed time.
		
		Consider a stopping time $\widehat\tau=\inf\{t\geq 1/2:X_t=0\}$ and note that $p=\p(\widehat\tau<3/4\wedge \tau_x)>0$, because an ub.v.\ process hits 0 immediately~\cite{bretagnolle,kesten}. For $\delta\leq 1/4$  we establish that 
		\begin{equation}\label{eq:bound_thm3}p\p(M>x,\underline M^\n<x,\tau<\delta)\leq \p(M>x,\underline M^\n<x,\tau\geq 1/2,D_\delta>x)\end{equation}
		 $D_\delta=\sup_{t\in(0,\delta],t\leq \tau-1/2}\{M-X_{\tau-t}\}$. The left hand side (by the strong Markov property) is the probability that our process hits $0$ in the interval $[1/2,3/4)$ (it has not yet crossed $x$) and in the following unit of time it achieves its supremum exceeding $x$ within $\delta$ time units, while the corresponding $\underline M^\n$ is below $x$, see Figure~\ref{fig:thm3}. It is not hard to see that this event implies the event on the right hand side ($\underline M^\n$ may become larger by means of the time interval $[0,1/2]$ but it must still be below $x$), and thus the inequality follows. It is crucial here  that the quantities do not depend on the grid shifting due to the random time~$\widehat\tau$.
		\begin{figure}[h!]
			\begin{center}
				\begin{tikzpicture}
				\draw (-1.2,0.5)--(-1,0.7)--(-0.8,-0.16)--(-0.6,0.34)--(-0.4,1.92)--(-0.2,1.77)--(0,0) -- (0.2,-0.05)--(0.4,-0.29)--(0.6,-0.88)--(0.8,1.5) -- (1,0.87)-- (1.2,0.2)--(1.4,0.5)--(1.6,2.5)--(1.8,1.67)--(2,1.32)--(2.2,1)-- (2.4,0.18)--(2.6,-0.81)--(2.8,-0.47)--(3,-0.34)--(3.2,0.3)--(3.4,1.26)--(3.6,1.45)--(3.8,1.2)--(4,0.7)--(4.2,1)--(4.4,1.15)--(4.6,1.16)--(4.8,0.17)--(5,0.5)--(5.2,0.64)--(5.4,0.9);
				\draw[->,thick] (-1.2,0.5)--(5.5,0.5);
				\draw[->,thick] (-1.2,-0.9)--(-1.2,3);
				\foreach \x in {0.8,1.4,1.8,2.8}   \draw[thick] (\x,0.6)--(\x,0.4);
				\draw (0.8,0.5) node[below]{$\frac{1}{2}$};
				\draw (1.8,0.5) node[below]{$\frac{3}{4}$};
				\draw (2.8,0.5) node[below]{$1$};
				\draw (1.4,0.5) node[below]{$\widehat \tau$};
				\draw[dashed] (-1.2,2.2)--(5.5,2.2);
				\draw[dashed] (-1.2,2.2) node[left]{$x$};
				\draw[dotted,<->] (1.4,-0.5)--(5.4,-0.5);
				\draw (3.4,-0.5) node[below]{$1$};
				\draw[dotted,<->] (1,2.5)--(1.6,2.5);
				\draw (1.3,2.5) node[above]{$\delta$};
				\end{tikzpicture}
			\end{center}
			\caption{Schematic sample path explaining the bound in~\eqref{eq:bound_thm3}}
			\label{fig:thm3}
		\end{figure}
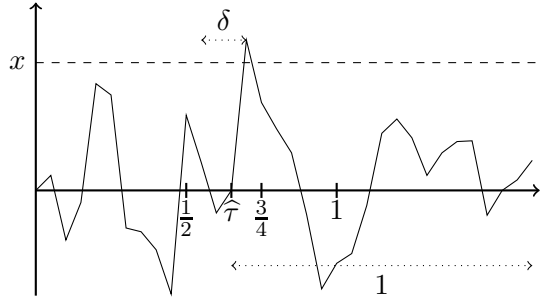
		
		The same arguments as in Proposition~\ref{prop:passage}  show that 
		\[b_n\p(M>x,\underline M^\n<x,\tau\geq 1/2,D_\delta>x)\to \e \underline \V \int f_{z,1/2}(x)\mu_{\delta,x}(\D z),\]
		where  $\mu_{\delta,x}(\D z)=\p(M\in\D z,\sup_{t\in(0,\delta],t\leq \tau}\{M-X_{\tau-t}\}>x)$.
		It is noted that now we are splitting the process at $1/2$, and the upper bound would be enough for what follows. Note also that restriction of $\underline M^\n$ to the times larger than $1/2$ makes it only smaller due to the inner maximum in its definition.
		
		Finally, since $f_{z,1/2}(x)$ is bounded for all $z>0$ according to Lemma~\ref{lem:density}, we find  that
		\[\limsup_{n\to\infty}b_n\p(M>x,M^\n<x,\tau<\delta)< C_x\p(\sup_{t\in(0,\delta],t\leq \tau}\{M-X_{\tau-t}\}>x)\]
		for some constant~$C_x$ not depending on~$\delta$. Moreover, the probability on the right hand side must decay to $0$ as $\delta\downarrow 0$,
		because ub.v.\ process does not jump at~$\tau$. The bound in~\eqref{eq:toshow} is now established and the proof is thus complete.
	\end{proof}
	
	\subsection{Further bounds and comments}
	It is noted that the above arguments may also be used to provide an asymptotic upper bound on the detection error in the case when~\eqref{eq:zooming} is not satisfied. Assuming that a non-monotone $X$ has a jointly continuous density $p(t,x)$ bounded for $|x|>\delta,t>0$ (e.g.~\eqref{eq:orey} holds with $\gamma>1$), we find with the help of Theorem~\ref{thm:moments_bound} that
	for any $\alpha_+>\alpha\vee 1$ and $\delta\in(0,1)$:
	\[\p(M>x,M^\n\leq x,\tau\geq \delta)=\Oh(n^{-1/\a_+})\qquad\text{ as }n\to\infty.\]
	Furthermore, we may strengthen this bound  to $\Oh(n^{-1/2})$ when $\a=2$, and to $\Oh(n^{-1})$ when $X$ is b.v.\ process, see Theorem~\ref{thm:moments_bound}.
	
	Moreover, we may also take $\delta=0$ apart from b.v.\ case when (i) $\gamma'=0$ or (ii) point $0$ is not in the support of~$\Pi(\D x)$, because in these cases the trick in the proof of Theorem~\ref{thm:passage} does not apply; there may be other ways to establish such bounds for these processes, however.
	
	There is also an asymptotic lower bound of order $n^{-1}$ on the detection error under some minor conditions. We omit the analysis of one-sided processes and state the following:
	\begin{lemma}\label{lem:bound_jumps}
		Assume that $X$ has jumps of both signs. Then
		\[\liminf_{n\to\infty} n\p(M>x,M^\n\leq x)>0.\]
	\end{lemma}
	\begin{proof}
		Let $\delta>0$ be a point of continuity of~$\overline \Pi_-(x)$ such that $\overline \Pi_-(\delta)>0$. Then $\p(X_t<-\delta)/t\to \overline \Pi_-(\delta)>0$ as $t\downarrow 0$, see e.g.~\cite{figueroa2008small}.
		Now consider a lower bound
		\begin{align*}&\p(M>x,M^\n\leq x)\\
		&\geq \p(\tau_x<1,X_{\tau_x}-x<\delta/2,\{\tau_x n\}<1/2,\sup_{t\in[\tau_x+1/(2n),1]}X_t<x)\\
		&\geq  \p(\tau_x<1,X_{\tau_x}-x<\delta/2,\{\tau_x n\}<1/2)\p(X_{1/(2n)}<-\delta)\p(M<\delta/2).
		\end{align*}
		Hence it is left to show that 
		\begin{equation}\label{eq:xxx}\liminf_{n\to\infty}\p(\tau_x<1,X_{\tau_x}-x<\delta/2,\{\tau_x n\}<1/2)>0.\end{equation}
		
		The compensation formula applied to the Poisson point process of jumps with intensity $\D t\times \Pi(\D y)$, see also~\cite{doney_kyprianou}, yields  
		\begin{multline}\p(\tau_x\in A,X_{\tau_x}\in(x,x+\delta/2))\\=\int_A \int_0^\infty\p(\overline X_{t}\leq x,X_{t}+y\in(x,x+\delta/2))\Pi(\D y)\D t\label{eq:compensation}\end{multline}
		showing that the corresponding measure is absolutely continuous.
		Hence $\{\tau_x n\}$ converges  weakly to a uniform random variable on the event $\tau_x<1,X_{\tau_x}\in(x,x+\delta/2)$; here we ignore the possibility of creeping over~$x$.
		Hence~\eqref{eq:xxx} is lower bounded by
		\[\p(\tau_x<1,X_{\tau_x}\in(x,x+\delta/2))/2,\]
		and it is left to show that this quantity is non-zero. Assume that $0$ belongs to the support of $\Pi$. Then using the ideas from~\cite[\S 24]{sato} we find that the support of $\p(\overline X_{t}\leq x,X_{t}\in \D x)$ is given by $(-\infty,x]$, and the positivity easily follows from~\eqref{eq:compensation}.
		The case when $0$ is not in the support of $\Pi(\D x)$ and there is no Brownian component corresponds to a possibly drifted compound Poisson.
		In view of Theorem~\ref{thm:passage} it is left to consider the latter, and in this case the statement follows from tedious but trivial considerations.
	\end{proof}

	\section*{Acknowledgments}
	We would like to thank Andreas Basse-O'Connor, Sonja Cox, Guido Lagos and Mark Podolskij for motivating discussions, and  
	Victor Rivero for letting us know about~\cite{doney_rivero}.
	The support of the grant
	`Ambit fields: probabilistic properties and statistical inference'  from Villum Foundation is greatfully acknowledged.
	The research of Krzysztof Bisewski is funded by the Netherlands Organisation for Scientific Research (NWO) through the research programme ‘Mathematics of Planet Earth’, grant number 657.014.003.

	\appendix
	
	\section{Remaining proofs}\label{app:proofs}
	\subsection{Proofs for Section~\ref{sec:prelim}}\label{app:s_prelim}
	We will need the following Karamata's theorem in the boundary case, see \cite[Thm.\ 1.5.9a--b]{bingham_regular}:
	\begin{lemma}\label{lem:karamata}
		Let $\ell(x)\in\RV_{0}$ (as $x\downarrow 0$), such that $\int_x^1\ell(t)\D t<\infty$ for any $x\in(0,1)$.
		\begin{itemize}
			\item[(i)]	Then $\int_x^1 t^{-1}\ell(t) \D t/\ell(x)\to\infty$ and the numerator is $\RV_0$.
			\item[(ii)] If $\int_0^1 t^{-1}\ell(t)\D t<\infty$ then $\int_0^x t^{-1}\ell(t) \D t/\ell(x)\to\infty$ and the numerator is $\RV_0$.
		\end{itemize}
	\end{lemma}
	%

	\begin{proof}[Proof of Proposition \ref{prop:zoom_sufficient}]
		Let us verify the conditions of~\cite[Thm.\ 2]{ivanovs_zooming}, which is trivial in the case $\a\in(0,1)\cup(1,2)$. 
		In case $\a=2$ our assumption implies that $\overline\Pi_-(x)\leq C\overline \Pi_+(x)$ for some finite $C$ and all $x$ small enough. Hence for such $x$ we have
		\[\frac{x^2\overline\Pi(x)}{v(x)}\leq (C+1)\frac{x^2\overline \Pi_+(x)}{\int_0^xy^2\Pi(\D y)}\]
		and it is left to show that this fraction converges to~0. Letting $\ell(x)=x^2\overline\Pi_+(x)\in\RV_0$ we find using integration by parts that 
		\[\int_0^xy^2\Pi(\D y)=-\int_0^xy^2\D\overline\Pi_+(y)=2\int_0^xy^{-1}\ell(y)\D y-\ell(x)+\ell(0+)\]
		showing that $\ell(0+)$ is convergent. According to Lemma~\ref{lem:karamata}(ii) we must have $\ell(0)=0$ and the above must explode when divided by $\ell(x)$ as $x\downarrow 0$. The proof in the case $\a=2$ is now complete.
		
		In the case $\a=1$ we need to show that $m(x)/(x\overline\Pi(x))\to \pm\infty$. Our assumption implies that $\overline \Pi_+(x)-\overline \Pi_-(x)>c\overline\Pi_+(x)$ for some $c>0$ and all small enough $x$, as well as $\overline\Pi(x)\leq 2\overline\Pi_+(x)$.
		We let $\ell_\pm(x)=x\overline \Pi_\pm(x)$ and first consider b.v.\ case. As above, observe that $\ell_+(0)=0$ and thus also $\ell_-(0)=0$. Now
		\begin{align*}
		m(x)&=-\int_0^x y\D\overline\Pi_+(y)+\int_0^x y\D\overline\Pi_-(y)\\
		&=\int_0^x y^{-1}(\ell_+(y)-\ell_-(y))\D y-(\ell_+(x)-\ell_-(x))\\
		&>c\int_0^x y^{-1}\ell_+(y)\D y-\ell_+(x)
		\end{align*}
		for small enough~$x$.
		Hence for small $x$
		\[\frac{m(x)}{x\overline\Pi(x)}>\frac{c\int_0^x y^{-1}\ell_+(y)\D y}{2\ell_+(x)}-1/2\to \infty\]
		according to Lemma~\ref{lem:karamata}(ii). This shows that~\eqref{eq:zooming} holds true with the limit process being a positive drift.
		
		In ub.v.\ case we have 
		\[\int_x^1y\Pi(\D y)=\int_x^1y^{-1}\ell(y)\D y-\ell(1)+\ell(x)\to \infty\]
		showing that $\int_x^1y^{-1}\ell(y)\D y\to\infty$, see Lemma~\ref{lem:karamata}(i).
		Now
		\begin{align*}
		m(x)&=\gamma-\int_x^1 y^{-1}(\ell_+(y)-\ell_-(y))\D y+(\ell_+(1)-\ell_-(1))-(\ell_+(x)-\ell_-(x))\\
		&<c'-c\int_x^1 y^{-1}\ell_+(y)\D y<-\frac{c}{2}\int_x^1 y^{-1}\ell_+(y)\D y
		\end{align*}
		for some constant $c'$ and small enough $x$, where the last line is implied by the divergence of the integral. Using Lemma~\ref{lem:karamata}(i) once again we find
		\[\frac{m(x)}{x\overline\Pi(x)}<-\frac{c/2\int_x^1 y^{-1}\ell_+(y)\D y}{2\ell_+(x)}\to -\infty\]
		and the proof is complete.
	\end{proof}
	
	\subsection{Proofs for Section~\ref{sec:ladder}}\label{app:s_ladder}
	
	\begin{lemma}\label{lem:index1}
		If $X$ is a driftless b.v.\ process attracted to a liner drift process under zooming-in then $\overline\Pi(\ep)\in\RV_{-1}$ as $\ep\downarrow 0$.
	\end{lemma}
	\begin{proof}
		From the proof of~\cite[Cor.\ 1]{ivanovs_zooming} we see that it must be that $M(x)=-\int_0^x y\D \overline\Pi(y)\in\RV_0$.
		Using integration by parts we get
		\[\overline\Pi(x)=\int_x^\infty x^{-1}\D M(x)=\int_x^\infty y^{-2}M(y)\D y-x^{-1}M(x).\] 
		By Karamata's theorem the last integral is in $\RV_{-1}$, and hence also the sum must be in $\RV_{-1}$.
	\end{proof}
	
	\begin{proof}[Proof of Proposition~\ref{prop:sup_tail_1}]
		First we show that when $X$ is a subordinator, then $\p(\overline X_1 \leq x)$ decays faster than any power of $x$. For that, see
		\begin{align*}
		& \p(X_1\leq x) \leq \p(X_{1/n}\leq x, X_{2/n}-X_{1/n}\leq x, \ldots, X_{n/n} - X_{(n-1)/n} \leq x) \\
		& \qquad = (1 - \p(X_{1/n} > x))^n = \Big[1 - \frac{n\p(X_{1/n} > x)}{n}\Big]^n \to e^{-\overline\Pi_+(x)},
		\end{align*}
		where the last limit holds for $x$, points of continuity of $\overline\Pi_+(\cdot)$, see e.g. \cite{figueroa2008small}. 
		We may assume that $\gamma'=0$ since otherwise the statement is obvious. Now $\Pi_+(x)\in\RV_{-\a}$, see also Lemma~\ref{lem:index1}, and we conclude that $\p(X_1\leq x)$ decays faster than any power.
		
		Now we proceed with the proof of the lower bound. From~\cite[Eq.\ (3.7)]{suprema_bounds} we have the inequality $\p(\overline X_t\leq x)\geq c_tU[0,x]$ with $c_t>0$, given that $\int_0 \kappa(1/s, 0)\D s<\infty$ and either  (i) $t$ is small enough or (ii) $X$ drifts to $-\infty$ implying $\kappa(0,0)>0$.
		Corollary~\ref{cor:kappa} and the assumption $\rho<1$ show that the integral is indeed convergent. Thus according to Proposition~\ref{prop:sup_tail} we have $\p(\overline X_t\leq x)\geq \RV_{\a\rho}$ when either (i) or (ii) hold. We carry out the proof in three distinct cases
		
		{\it Case 1.} $\Pi(\R_-)>0$ and $\rho<1$. Choose $h>0$ such that $\overline \Pi_-(h)=\lambda>0$. Let $t$ be small enough so that the above bound is true for the process $X^h$ with negative jumps exceeding $h$ in absolute value removed ($X^h$ belongs to the same class $\mathcal D_{\a,\rho}$ as $X$, see~\cite[Lem.\ 3]{ivanovs_zooming}). Now by requiring that a big negative jump occurs before $t$ we get a lower bound:
		\[\p(\overline X_1\leq x)\geq \p(\overline X^h_{e_\lambda}\leq x,e_\lambda<t)\p(\overline X_1\leq h)\geq \p(\overline X^h_t\leq x) c\]
		with $c>0$, where $e_\lambda$ is the first time of the big negative jump. The result now follows when $\rho\neq 1$.
		
		{\it Case 2.} $\Pi(\R_-)>0$ and $\rho=1$. Note that $\p(\overline X^h_\ep/a_\ep\leq c)\to \p(\overline\X_1\leq c)>0$ for $c>0$ large enough (here large $c$ is needed in case $\X$ is a drift process). 
		Take the asymptotic inverse $a_u^{-1}\in\RV_{\a}$ of $\a_\ep$~\cite[Thm.\ 1.5.12]{bingham_regular} to find that $\p(\overline X^h_{a^{-1}_{x/c}}\leq x)$ has a positive limit.  It is left to require a jump of size $<-h$ in the time interval $(0,a^{-1}_{x/c})$ to get a lower bound (up to a positive constant). But the probability of such jump is $\lambda a^{-1}_{x/c}\in\RV_{\a}$ as claimed.
		
		{\it Case 3.} $\Pi(\R_-)=0$. Assume that $\overline \Pi_+(h)=\lambda>0$ for some $h>0$, and let $X^h$ be the process with positive jumps exceeding $h$ removed.
		We have the lower bound 
		\[\p(\overline X_1\leq x)\geq \p(\overline X^h_1\leq x,e_\lambda>1)=\p(\overline X^h_1\leq x)e^{-\lambda},\]
		which is $\RV_{\a\rho}$ when $X^h_t\to-\infty$ as $t\to\infty$. 
		That is, we need to ensure that $\e X^h_1=\gamma-\int_h^1 x\Pi(\D x)<0$ with $h<1$.
		This is always possible when $\int_0^1 x\Pi(\D x)=\infty$, and is also possible when $\int_0^1 x\Pi(\D x)<\infty$ and $\gamma'<0$.
		
		It is left to consider an independent sum of a linear Brownian motion and a pure jump subordinator.
		We may also represent such a process as $W_t+Y_t$ with a linear Brownian motion $W$ and a bounded variation process $Y$ with linear drift $\gamma_Y'<0$.
		Clearly, 
		\[\p(\overline X_1\leq x)\geq \p(\overline W_1\leq x/2)\p(\overline Y_1\leq x/2)\]
		where $\p(\overline Y_1\leq x/2)$ has a positive limit. Thus we only need to show that $\p(\overline W_1\leq x)=\p(\tau^W_x>1)\geq\RV_1$, where $\tau_x^W$ is an inverse Gaussian subordinator. But this is immediate and the proof is now complete.
	\end{proof}
	
	\begin{proof}[Proof of Proposition~\ref{prop:bound_excursion}]
		Using the strong Markov property we get for any $s\in(0,\ep)$
		\begin{align*}\underline n(X_\ep>\delta) & =\int_0^\infty \underline n(X_s\in\D x)\p(X_{\ep-s}>\delta-x,\underline X_{\ep-s}> -x)\\
		&\leq \underline n(X_s>\delta/2)+\underline n(X_s\leq \delta/2 )\p(X_{\ep-s}>\delta/2)\\
		&\leq \underline n(X_s>\delta/2)+\underline n(s<\zeta)\p(\overline X_{\ep}>\delta/2).
		\end{align*}
		Using~\cite[Thm.\ 6]{chaumont_supremum} and the above lower bound on $\underline n(X_s>\delta/2)$ we find
		\begin{align*}&\p(\overline X_\ep>\delta/2)=\int_0^\ep \underline n(X_s>\delta/2)n(\ep-s<\zeta) \D s\\
		&\geq \underline n(X_\ep>\delta)\int_0^\ep n(\ep-s<\zeta) \D s-\p(\overline X_\ep>\delta/2)\int_0^\ep \underline n(s<\zeta)n(\ep-s<\zeta) \D s,
		\end{align*}
		where the latter integral equals to~1, and thus we have
		\[\underline n(X_\ep>\delta)\int_0^\ep n(s<\zeta) \D s\leq 2\p(\overline X_\ep>\delta/2).\]
		But $\int_0^\ep n(s<\zeta) \D s\in\RV_{1-\rho}$,
		and so it is left to observe that $\p(\overline X_\ep>\delta/2)=\Oh(\ep)$, which readily follows from Doob-Kolmogorov inequality.

		Similarly, we have
		\[\underline n(X_\ep>\delta)\geq \underline n(X_s>2\delta)\p(\underline X_\ep>-\delta)\]
		and then also
		\[\p(\overline X_\ep>2\delta)\leq\underline n(X_\ep>\delta)\int_0^\ep n(\ep-s<\zeta) \D s/\p(\underline X_\ep>-\delta)\]
		yielding the lower bound, since $\p(\overline X_\ep>2\delta)\geq \p(X_\ep>2\delta)\sim \ep\overline\Pi_+(2\delta)$.
	\end{proof}
	
	\subsection{Proofs for Section~\ref{sec:moments}}\label{app:s_moments}
	
	Recall that $(\gamma^\n,\sigma^\n,\Pi^\n)$ is the L\'evy triplet of the rescaled process $b_nX_{t/n}$ as defined in~\eqref{eq:n_triplet}.
	\begin{proof}[Proof of Lemma~\ref{lem:levyp_conv}]
		Using $\Pi^\n(\D x)=\Pi(b_n^{-1}\D x)/n$ we find that 
		\[\int_1^\infty x^p\Pi^\n(\D x)=\frac{b_n^p}{n}\int_{b_n^{-1}}^\infty x^p\Pi(\D x).\]
		Since $b_n$ is regularly varying at $\infty$ with index $1/\a$ we see that $b_n^p/n\to 0$ and so it is sufficient to consider the limit of  
		\[\frac{b_n^p}{n}\int_{b_n^{-1}}^1 x^p\Pi(\D x)=\frac{\ep}{a_\ep^p}\int_{a_\ep}^1 x^p\Pi(\D x),\]
		where $\ep=1/n\downarrow 0$.		
		Recall also the definitions of $m(x)$ and $v(x)$, the truncated mean and variance functions, given in~\eqref{def:trunc_mv}.

		
		Consider the case where $\X$ is a Brownian motion, so that $p<2$. 
		From~\cite[Thm.\ 6 (i)]{ivanovs_zooming} we see that $v\in\RV_0$ and $\ep v(a_\ep)/a_\ep^2\to\widehat \sigma^2$. 
		Hence 
		\[\frac{\ep}{a_\ep^p}\int_{a_\ep}^1 x^p\Pi(\D x)\leq\frac{\ep}{a_\ep^p}\int_{a_\ep}^1 x^{p-2} \D v(x)=\frac{\ep v(a_\ep)}{a_\ep^2} \cdot \tfrac{\int_{a_\ep}^1 x^{p-2} \D v(x)}{a_\ep^{p-2}v(a_\ep)}\to 0,\]
		because the first ratio converges to $\widehat\sigma^2$ and the second to 0 according to the Karamata's theorem, see~\cite{bingham_regular} or~\cite[Thm.\ 6]{ivanovs_zooming}.
		
		Consider the case of a strictly $\a$-stable process $\X$. Let $f_+(x) := \Pi(x,1)$, $f_-(x) := \Pi(-1,-x)$ and $f(x) = f_-(x)+f_+(x)$ then according to~\cite[Thm.\ 2]{ivanovs_zooming} we have $f_\pm(x)\in\RV_{-\a}$ (or at least the dominating one) and $\ep f_\pm(\a_\ep)\to\tfrac{\widehat c_\pm}{\a}$. Since
		\[\int_x^1 y^p \Pi(\D y) = x^p f_+(x) + p\int_x^1 y^{p-1} f_+(y)\D y\]
		thus we have
		\begin{align*}
		& \frac{\ep}{a_\ep^p}\int_{a_\ep}^1 x^p\Pi(\D x) = \frac{\ep}{a_\ep^p} \bigg(a_\ep^p f_+(a_\ep) + p\int_{a_\ep}^1 x^{p-1} f_+(x)\D x\bigg) \\
		& \quad = \ep f_+(a_\ep) \cdot \left( 1 + \tfrac{p\int_{a_\ep}^1 x^{p-1} f_+(x)\D x}{a_\ep^p f_+(a_\ep)}\right) \to \frac{\widehat c_+}{\a} \cdot \bigg( 1 + \frac{p}{\a-p}\bigg)=\frac{\widehat c_+}{\a-p}
		\end{align*}
		and the result follows.
		
		Finally, consider the case, where $\X$ is a linear non-zero drift process. Then necessarily $\a=1$ and according to~\cite[Thm.\ 6 (ii)]{ivanovs_zooming} we must have 
		$x\overline\Pi(x)/m(x)\to 0$ and $\ep m(a_\ep)/a_\ep\to\widehat\gamma$. Letting $M(x)=\int_{x\leq |y|<1}|y|^p\Pi(\D y)$ note that it is sufficient to show that $\ep M(a_\ep)/a_\ep^p\to 0$.
		
		Let $f(x) = f_+(x) + f_-(x)$. The main difficulty here is that $f(x)$ is does not necessarily belong to the class $\RV_{-1}$ however we do have that $m\in\RV_0$ according to \cite[Proof of Thm.\ 6]{ivanovs_zooming}, see also its proof. We have $xf(x)/m(x)\to 0$ and $\ep m(a_\ep)/a_\ep\to\widehat\gamma$ thus $\ep f(a_\ep)\to 0$ and for any $\delta>0$ there exists $x_0$ such that $xf(x)\leq \delta m(x)$ for $x<x_0$. Then
		\begin{align*}
		& \frac{\ep}{a_\ep^p}\int_{a_\ep\leq|x|<1} |x|^p\Pi(\D x) = \frac{\ep}{a_\ep^p} \bigg(a_\ep^p f(a_\ep) + p\int_{a_\ep}^1 x^{p-1} f(x)\D x\bigg) \\
		& \quad = \ep f(a_\ep) + \frac{\ep m(a_\ep)}{a_\ep} \cdot \frac{p\int_{a_\ep}^{x_0} x^{p-2} m(x)\D x}{a_\ep^{p-1} m(a_\ep)} \cdot \delta + \frac{\ep}{a_\ep^p}\int_{x_0}^1 x^{p-1}f(x)\D x
		\end{align*}
		The first and the third term converge to $0$. Since $m\in\RV_0$, then according to Karamata's Theorem we have $\frac{\int_{a_\ep}^{x_0} x^{p-2} m(x)\D x}{a_\ep^{p-1} m(a_\ep)} \to \frac{1}{1-p}$ and since the choice of $\delta>0$ was arbitrary, we conclude that
		\[\frac{\ep}{a_\ep^p}\int_{a_\ep\leq|x|<1} |x|^p\Pi(\D x) \to 0,\]
		and the proof is complete.
	\end{proof}
	
	Next, we consider the general case where~\eqref{eq:zooming} does not necessarily hold, but redefine $b_n=n^{1/\a_+}$ for some $\a_+>\a$, and thus also $X^\n_t=b_n X_{t/n}$.
	\begin{lemma}\label{lem:gen_conv0}
		For any $\a_+>\a$ and $b_n=n^{1/\a_+}$ we have $X_1^\n\convd 0$. If, moreover, $p<\a_+$ and $\int_1^\infty x^p\Pi(\D x)<\infty$ then $\int_1^\infty x^p\Pi^\n(\D x)\to 0$ as $n\to\infty$.
	\end{lemma}
	\begin{proof}
		First, we consider the integral.
		As in the beginning of the proof of Lemma~\ref{lem:levyp_conv} it is sufficient to observe that
		\[\tfrac{b_n^p}{n}\int_{1/b_n}^1 x^p\Pi(\D x) \leq \tfrac{1}{n}\int_{1/b_n}^1 (b_n x)^{\beta_+}\Pi(\D x) \leq \tfrac{b_n^{\beta_+}}{n}\int_0^1 x^{\beta_+}\Pi(\D x)\to 0,\]
		where $\beta_+\in(\a\vee p,\a_+)$ and the latter integral is finite because $\beta_+>\a\geq \beta_0$.
		
		According to~\cite[Thm.\ 15.14]{kallenberg} the convergence $X_1^\n\convd 0$ is equivalent to $m(b_n^{-1})b_n/n,v(b_n^{-1})b^2_n/n,\overline\Pi_\pm(ub_n^{-1})/n\to 0$ for all $u>0$. All of these limits can be shown using the above trick, and we only consider the first quantity (the most tedious).
		If $\int_{|x|<1} |x|\Pi(\D x)<\infty$ then $m(b_n^{-1})=\gamma'+\int_{|y|<b_n^{-1}}y\Pi(\D y)$. The case $\a=1$ is trivial, but for $\a<1$ we have
		\[\frac{b_n}{n}\int_{|y|<b_n^{-1}}|y|\Pi(\D y)\leq \frac{1}{n}\int_{|y|<b_n^{-1}}|b_ny|^{\beta_+}\Pi(\D y)\leq \frac{b_n^{\beta_+}}{n}\int_{|y|<1}|y|^{\beta_+}\Pi(\D y)\to 0\]
		with $\beta_+\in(\a,\a_+\wedge 1)$. When $\int_{|x|<1} |x|\Pi(\D x)=\infty$ we have $\a\geq 1$ and it is sufficient to note that
		\[\frac{b_n}{n}\int_{b_n^{-1}<|y|<1}|y|\Pi(\D y)\leq \frac{b_n^{\beta_+}}{n}\int_{|y|<1}|y|^{\beta_+}\Pi(\D y)\to 0\]
		for $\beta_+\in(\a,\a_+)$.
		The proof is concluded.
	\end{proof}
	
	\begin{lemma}\label{lem:gen_conv2}
		In the cases (i) $p\leq \a=2$ and (ii) $p\leq \a=1$ with $X$ b.v., the sequences 
		$|\gamma^\n|, \sigma^\n,\int_{|x|\leq 1} x^2\Pi^\n(\D x),\int_1^\infty x^p\Pi^\n(\D x)$ for the scaling $b_n=n^{1/\a}$ are bounded provided that $\int_1^\infty x^p\Pi(\D x)<\infty$.
	\end{lemma}
	\begin{proof}
		The statement is trivially true for $\sigma^\n$ and the rest follows using the simple trick from Lemma~\ref{lem:gen_conv0}. We only consider 
		\[\int_1^\infty x^p\Pi^\n(\D x)=\frac{1}{n}\int_{n^{-1/\a}}^\infty (n^{1/\a}x)^p\Pi(\D x)\leq C+\frac{1}{n}\int_{n^{-1/\a}}^1 (n^{1/\a}x)^\a\Pi(\D x),\]
		where we used $p/\a-1\leq 0$ and convergence of $\int_1^\infty x^p\Pi(\D x)$. The second term converges to $\int_0^1 x^\a\Pi(\D x)<\infty$, which is finite in both cases (i) and~(ii).
	\end{proof}

	
	\begin{proof}[Proof of Proposition~\ref{prop:UI_onejump}]
		It is clear that $\indevent_{{A^\n}^c}$ converges to~$1$ in probability and so by Slutsky's Lemma we find that $\widetilde V^\n\convd \V$. Thus it is left to show uniform integrability.
		
		Conditional on the event $\{N=k,{A^\n}^c\}$, split the process into $k+1$ pieces separated by the big jumps, where each piece starts at 0 and does not include the terminating big jump. Let $V^\n_{i,k}$ be the (conditional) discretization error for the supremum of the $i$-th piece keeping the original grid.
		Since each piece is at least $1/n$ long and conditioning affects only the length of the piece, the same argument as in the proof of Theorem~\ref{thm:UI} based on Lemma~\ref{lem:bertoin} shows that the $1+\epsilon$ moment of $V^\n_{i,k}$ is bounded by a constant $C$ not depending on $i,k$. Note that by removing big jumps we still have a process in $\mathcal D_{\a,\rho}$, but now $\beta_\infty=\infty$.
		
		It is left to note that the conditional $V^\n$ is bounded by the maximum over $V_{i,k}^\n$ which in turn is bounded by the sum.
		Hence using Minkowski's inequality we find
		\begin{align*}
		\e \left( \widetilde V^\n\right)^{1+\epsilon} & \leq \sum_{k=0}^\infty \p(N=k)\e \left(\left( V^\n\right)^{1+\epsilon}|{A^\n}^c,N=k\right)\\
		& \leq \sum_{k=0}^\infty \p(N=k) (k+1)^{1+\ep} C \leq \e (1+N)^{1+\ep} C<\infty
		\end{align*}
		and the proof is complete.
	\end{proof}

	\begin{proposition}\label{prop:cpp}
		Let $X$ be a compound Poisson process with drift $\gamma\in\R$ satisfying $\int_{|x|>1}|x|\Pi(\D x) < \infty$. Then~\eqref{eq:cpp} holds true.
	\end{proposition}
	\begin{proof}
		For $n\in\N$, $k\in\{0,\ldots,n-1\}$ define
		\begin{align*}
		N_{k,n} := \# \left\{t\in\left(\tfrac{k}{n},\tfrac{k+1}{n}\right) : X_{t-}\neq X_t\right\}, \quad A_{k,n} := \{N_{k,n} \leq 1\}
		\end{align*}
		and put $A_n := \bigcap_{k=0}^{n-1} A_{k,n}$. We have the decomposition
		\[n\e(M-M^\n) = n\e(M-M^\n; A_n) + n\e(M-M^\n; A_n^c).\]
		
		{\it Step 1.} Show that $n\e(M-M^\n; A_n)\to \tfrac{1}{2}|\gamma|\p(\tau\in(0,1))$. This is clear when $\gamma=0$; in the following we assume $\gamma\neq0$. Since then $X$ is in a domain of attraction of linear drift, it follows that
		\[\Big(n(M-M^\n) \mid \tau\in(0,1)\Big) \convd  \V \eqd |\gamma|U,\]
		where $U$ is uniformly distributed over $[0,1]$. Combination of Slutsky Lemma with the uniform integrability of $(M-M^\n)\ind{A_n}$ yields the result.
		
		{\it Step 2.} Show that, with $B_k := \{N_{k,n}=2,\cap_{i\neq k} A_{i,n}\}$, the following are equal up to $\oh(1)$ term:
		\begin{equation}\nonumber
		\begin{split}
		& n\e(M-M^\n;A_n^c),\quad \textstyle\sum_{k=0}^{n-1}n\e(M-M^\n;A_{k,n}^c),\\
		& \textstyle\sum_{k=0}^{n-1}n\e(M-M^\n;N_{k,n}=2),\quad  \textstyle\sum_{k=0}^{n-1}n\e(M-M^\n;B_k).
		\end{split}
		\end{equation}
		This step is a rather tedious, but also, a pretty straightforward application of inclusion-exclusion principle. We only show the first equivalence, as the rest is similar. Note that $\p(A_{k,n})= O(n^{-2})$ and that we have a very crude upper bound $M-M^\n \leq \gamma + \sum_{k=1}^N |J_k|$, where $N$ is the number of jumps of CP process and $J_1,J_2,\ldots$ are iid jumps. For $j<k<n$ we have
		\begin{align*}
		\e(M-M^\n;A_{j,n}^c\cap A_{k,n}^c)  \leq \p(A_{j,n}^c\cap A_{k,n}^c)\e\Big(\gamma + \textstyle\sum_{k=1}^N |J_k| \,\big|\, A_{j,n}^c\cap A_{k,n}^c\Big) \\
		= \p(A_{k,n})^2\e\Big(\gamma + \textstyle\sum_{k=1}^N |J_k| \,\big|\, N\geq 4\Big) \leq Cn^{-4},
		\end{align*}
		where $C>0$ does not depend on $j,k,n$. This implies that
		\[\textstyle\sum_{0\leq j<k<n}n\e(M-M^\n;A_{j,n}^c\cap A_{k,n}^c) \to 0.\]
		
		{\it Step 3.} Notice that when $X$ is a Compound Poisson process then $I$ has an alternative representation:
		\[I = \lambda^2\e\Big((J_1 + \underline X_{U})\wedge (-J_2 - \overline X'_{1-U})\Big)^+,\]
		where $\lambda = \Pi(\R)$, $U$ is uniformly distributed over $[0,1]$, $X'$ is a statistical copy of $X$, random variables $J_1,J_2$ have the law $\Pi(\D x)/\lambda$, and $U,X,X',J_1,J_2$ are independent. 
		
		{\it Step 4.} Show that $\sum_{k=0}^{n-1}n\e(M-M^\n;B_k)\to \tfrac{1}{2}I$. Working on the event $B_k$, let $J_1,J_2$ be the two jumps in time interval $\left(\tfrac{k}{n},\tfrac{k+1}{n}\right)$ in the order of occurrence. Moreover, let
		\[L_t := X_t - \textstyle\sup_{s\in[0,t]} X_s, \quad R_t := \textstyle\sup_{t\in[s,1]} X_s - X_t\]
		and notice that $L_{t_1},R_{t_2}$ are independent when $t_1\leq t_2$. We have
		\begin{align*}
		(M-M^\n)\ind{B_k} \geq \left(\big((J_1+L_{k/n})\wedge(-J_2-R_{(k+1)/n})\big)^+ - |\gamma|/n\right)\ind{\cap_{i\neq k}A_{i,n}}
		\end{align*}
		and an analogous upper bound holds true, with $+|\gamma|/n$ instead of $-|\gamma|/n$. Now, we denote
		\begin{align*}
		G^{(n-)}(t) & := \e \left(\big((J_1+L_t)\wedge (-J_2- R_t)\big)^+ - |\gamma|/n; A_n\right) \\
		G^{(n+)}(t) & := \e \big((J_1+L_t)\wedge (-J_2-R_t)\big)^+ + |\gamma|/n.
		\end{align*}
		$L_t$ and $R_t$ are stochastically non-increasing and non-decreasing respectively since $L_t \eqd \underline X_t$, $R_t\eqd \overline X_{1-t}$ (this holds true also on the event $A_n$) thus
		\[G^{(n-)}((k+1)/n) \leq (M-M^\n)\ind{B_k} \leq G^{(n+)}(k/n).\]
		It is clear that $G^{(n\pm)}(t)\to G(t)$ point-wise, where $G(t) := \e \big((J_1+\underline X_t)\wedge (-J_2-\overline X'_{1-t})\big)^+$. Since $\p(N_{k,n}=2) = \frac{\lambda^2}{2n^2}e^{-\lambda/n}$, we have
		\begin{align*}
		&\sum_{k=1}^{n-1} n\e(M-M^\n;B_k) \leq \sum_{k=0}^{n-1}n\p(N_{k,n}=2)G^{(n+)}(k/n) \\
		& \qquad \qquad= \tfrac{\lambda^2}{2}\sum_{k=0}^{n-1} \tfrac{1}{n}G^{(n+)}(k/n) \to \tfrac{\lambda^2}{2}\int_0^1 G(t)\D t = \tfrac{1}{2} I,
		\end{align*}
		where we used dominated convergence. Analogous reasoning leads to the same lower bound, which concludes the proof.
	\end{proof}
	
	
	\subsection{Proofs for Section~\ref{sec:detection_error}}\label{app:detection_error}
	
	\begin{proof}[Proof of Lemma~\ref{lem:bounded_density}]
		Letting $\phi(\theta)=\psi(\ii \theta)$, we note that the condition~\eqref{eq:orey} ensures the following bound on the characteristic function of $X_t$: $|e^{\phi(\theta)t}|\leq \exp(-ct|\theta|^\gamma)$ for some $c>0$ and $|\theta|>1$, see~\cite[Lem.\ 2.3]{picard}.
		By the inversion formula we have
		\[p(t,x)=\frac{1}{2\pi}\int_\R e^{-\ii x\theta+\phi(\theta)t}\D \theta,\]
		because the characteristic function $e^{\phi(\theta)t}$ is integrable.
		Thus $p(t,x)$ is bounded for all $t>\delta,x\in\R$, and so we need to consider $t\in(0,\delta],x>\delta$ since the case $x<-\delta$ is analogous.
		
		Assume for a moment that $X$ has no jumps larger than $1$ in absolute value, and so $\phi(\theta)$ is smooth.	
		From the L\'evy-Khintchine formula we find that $ |\phi'(\theta)|\leq c_0+c_1|\theta|$ and $|\phi^{(k)}(\theta)|\leq c_k$ for $k\geq 2$ and some positive constants $c_k$; for this we differentiated under the integral with respect to $\Pi(\D x)$ and used the inequality $|e^{\ii a}-1|\leq |a|$.
		Integration by parts gives
		\[\int_0^\infty e^{-\ii x\theta+\phi(\theta)t}\D \theta=
		\frac{1}{\ii x}+\int_0^\infty \frac{1}{\ii x}\phi'(\theta)te^{-\ii x\theta+\phi(\theta)t}\D \theta,\]
		and it would be sufficient to establish that  
		\[\int_1^\infty t(c_0+c_1\theta) \exp(-c\theta^\gamma t)\D \theta\]
		is bounded for all $t\in(0,\delta)$. This, however, is only true for $\gamma=2$. Nevertheless, we may apply integration by parts $k$ times to arrive at the bound:
		\[\int_1^\infty A(t,\theta) \exp(-c\theta^\gamma t)\D \theta,\]
		where $A(t,\theta)$ is a weighted sum of the terms $\theta^it^j$ with $i<j$ and $i=j=k$; one may also use Fa\`a di Bruno's formula here.
		Note that 
		\[\int_0^\infty \theta^it^j \exp(-c\theta^\gamma t)\D \theta=t^{j-(i+1)/\gamma}\int_0^\infty \theta^i\exp(-c\theta^\gamma)\D \theta,\]
		which is bounded for small $t$ when $\gamma\geq (i+1)/j$. Since $\gamma>1$ this inequality is always satisfied for the integers $i<j$, whereas for $i=j=k$ we get $\gamma\geq (k+1)/k$ and so we simply need to ensure that $k$ is sufficiently large.
		
		Suppose now that $X_t=\widehat X_t+Y_t$ is an independent sum, where $Y$ is a Poisson process with jumps larger than $1$ . The density of $X_t$ is given by
		\[p(t,x)=\int \p(Y_t\in\D z) \widehat p(t,x-z)\leq \p(Y_t=0)\widehat p(t,x)+\p(Y_t\neq 0)\sup_x\widehat p(t,x),\]
		where $\widehat p(t,x)$ is bounded on the set away from the origin. It is thus sufficient to show that
		the second term  stays bounded as $t\downarrow 0$. But $\p(Y_t\neq 0)$ is of order $t$ and $\sup_x\widehat p(t,x)=\O(t^{-1/\gamma})$ according to~\cite[Thm.\ 3.1]{picard} completing the proof.
		
		Finally, suppose that~\eqref{eq:orey} is satisfied with $\gamma<1$ but for some $\gamma'\in(\gamma,1)$, we have that
		\[\lim_{\epsilon\to 0}\epsilon^{\gamma'-2}\int_{-\epsilon}^\epsilon x^2\Pi(\D x)=0\]
		which according to~\cite[Thm.\ 3.1(b)]{picard} implies $\sup_x p(t,x)\geq c t^{-1/\gamma'}$ for $t$ small enough.
		We may assume that for small enough $t$ the supremum is achieved by $x\in[-\delta,\delta]$, because otherwise we have a contradiction.
		Now suppose that $\Pi(\D x)$ has a point mass at $1$, so that with probability of order $t$ there is one jump of size~$1$. 
		But then $\sup_{x\in[1-\delta,1+\delta]} p(t,x)\geq c_1 t^{1-1/\gamma'}\to \infty$ as $t\to 0$ showing that $p(t,x)$ explodes away from~$x=0$.
	\end{proof}
	
	\section{Correction}\label{app:correction}
	The results of this paper build on~\cite{ivanovs_zooming} which, however, has a mistake and a gap in the proofs, and we correct them in the following. 

	\medskip	
	Firstly, we note that convergence of transforms
	\[\int_0^\infty e^{-qt}f_n(t)\D t\to \int_0^\infty e^{-q t}f(t)\D t\qquad\text{ for all }q>0\]
	for bounded non-negative $f_n,f$ does not imply that $f_n(t)\to f(t)$ for almost all $t>0$ (it is true for the cumulative functions $\int_0^t f_n(s)\D s$).
	This renders the last paragraph in the proof of Theorem~4 in~\cite{ivanovs_zooming} invalid, and hence we only have the limit theorem for a killed L\'evy process, that is, when the time horizon $T$ is assumed to be an independent exponential random variable. 
	Nevertheless, we may use the fact that the convergence is R\'enyi-mixing to provide a simple extension to a deterministic~$T$.
	
	Consider the process $X$ on the time interval $[0,T]$ and let $M,\tau$ denote the supremum and its time. Let $F^\epp_T$ be a bounded functional of $(X_{\tau+s\ep}-M)_{s\in[-r,r]}$ for some fixed number $r>0$, and let $A$ and $B$ be events in $\sigma(X_s,s\in[0,1])$ and in $\sigma(X'_s,s\geq 0)$, respectively, where $X'_s=X_{1+s}-X_1$. In the following we assume that $T$ is an independent exponential.  It is established in~\cite{ivanovs_zooming} that 
	\begin{equation}\label{eq:mistake}\e (F^\epp_T;A|T\geq 1,B)\to f\p(A|T\geq 1,B)=f\p(A)\end{equation}
	as $\ep\downarrow 0$. Here $\p(T\geq 1,B)>0$ and $f$ does not depend on the choice of~$A$ and $B$, since we have R\'enyi-mixing convergence for a killed original process. 
	
	It is only required to prove that $\e (F^\epp_1;A)\to f\p(A)$.
	Assume for a moment that $A$ implies $\overline X_1\neq X_1$ and choose 
	\[A'=A\cap\{\overline X_1=\overline X_{1-\delta},\overline X_1-X_1>h\},\quad B'=\{\overline X'<h\},\]
	so that on $T\geq 1$ the event $A'\cap B'$ implies that $\tau<1-\delta$ and hence for $\ep<\delta/r$ we have $F^\epp_T=F^\epp_1$. Thus
	\[\e (F^\epp_1;A')=\e (F^\epp_1;A'|T\geq 1,B')=\e (F^\epp_T;A'|T\geq 1,B')\to f\p(A').\]
	Moreover, we have the bound
	\[\sup_{\ep>0}|\e (F^\epp_1;A)-\e (F^\epp_1;A')|\leq C\p(\overline X_1\neq \overline X_{1-\delta}\text{ or }\overline X_1-X_1\leq h)\to 0\]
	as $\delta,h\downarrow 0$, since we have assumed that the supremum over $[0,1]$ is not attained at~$1$.
	This shows 
	\[\e (F^\epp_1;A,\overline X_1\neq X_1)\to f\p(A,\overline X_1\neq X_1)\]
	for an arbitrary event $A\in\sigma(X_s,s\in[0,1])$.
	When $\p(\overline X_1=X_1)>0$, we may look at the time-reversed process on the event that it does not become negative. The result in this case is, in fact, trivial and the proof of Theorem~4 in~\cite{ivanovs_zooming} is fixed.
	
	\medskip
	Secondly, there is a problem with the continuity of some basic functionals of sample paths on infinite time intervals. In particular, the supremum and its time are not continuous on the Skorokhod space $D[0,\infty)$,
	even though they are obviously continuous on the relevant subsets of $D[0,T]$ for any $T>0$.
	This concerns the proof of Theorem 5 as well as the proof of Theorem 3 in Appendix of~\cite{ivanovs_zooming}.
	The main tool to overcome this issue is the approximation Lemma, see~\cite[Thm.\ 3.2]{billingsley} or~\cite[Thm.\ 4.28]{kallenberg}.
	
	In the following we provide a correction to the proof of Theorem 5 in~\cite{ivanovs_zooming}, and note that the same bounds can be used with respect to Theorem 3. 
	More concretely, we need to show for any $a>0$ that
	\begin{equation}\label{eq:correction2}\lim_{T\to\infty}\limsup_{\ep\downarrow 0}\p(\sup_{t\geq T}{Y^\epp_t}\geq - a)= 0,\end{equation}
	where $Y^\epp_t=(X_{\tau+t\ep}-M)/a_\ep$ is the  post-supremum process of $X^\epp$ corresponding to time-space rescaling of~$X$; similar statement is needed for the pre-supremum process, see also~\cite[Lem.\ 4]{asmussen_glynn_pitman1995}.
	
	Note that it is sufficient to prove~\eqref{eq:correction2} for $X$ killed at an independent exponential time instead of time~$1$, and we assume this in the following. Next, we use the Markov property to see that 
	\begin{equation}\label{eq:limsup}\p(\sup_{t\geq T}{Y^\epp_t}\geq -a)=\int_{-\infty}^0 \p_x(\sup_{t\geq 0}{Y^\epp_t}\geq -a)\p(Y^\epp_T\in \D x),\end{equation}
	where $Y^\epp$ under $\p_x$ is the (killed) process conditioned to stay negative, and the measures $\p(Y^\epp_T\in \D x)$ have a proper weak limit as $\varepsilon\downarrow 0$. Furthermore, for $x>a$ we have
	\[\p_{-x}(\sup_{t\geq 0}{Y^\epp_t}\geq -a)=1-m^\epp(x-a)/m^\epp(x)\]
	with $m^\epp(x)=\e\int_0^\infty\ind{H^\epp_t<x}\D t$ analogously to~\cite[Thm.\ 1]{chaumont} establishing this formula for the limit process. According to the proof of~\cite[Thm.\ 3]{ivanovs_zooming} $m^\epp(x)\to m(x)$ for all $x>0$ with $m$ corresponding to the self-similar limiting process. Moreover, the functions $m^\epp,m$ are monotone and continuous, and so $m^\epp(x_\ep)\to m(x)$ whenever $x_\ep\to x>0$. Hence~\eqref{eq:limsup} converges as $\ep\downarrow 0$ to the respective probability for the limit post-supremum process. The result follows upon letting $T\to\infty$.
\end{document}